\documentclass[a4paper,10pt]{article}

\usepackage{hyperref}
\usepackage[utf8]{inputenc}
\usepackage[english]{babel}
\usepackage[babel]{csquotes}
\usepackage{fullpage}
\usepackage{amsfonts, amsthm, amsmath,amssymb}
\usepackage{enumerate}
\usepackage{adjustbox}
\usepackage{stackrel}
\usepackage{tablefootnote}
\usepackage{centernot}
\usepackage{bbm}       
\usepackage[version=3]{mhchem}    
\usepackage{tikz}      
\usetikzlibrary{shapes,automata,positioning,arrows,fit}
\usepackage{color}      
\usepackage{mathtools}  
\usepackage{cite}

\usepackage{caption}
\numberwithin{equation}{section}

\newtheorem{theorem}{Theorem}[section]
\newtheorem{corollary}[theorem]{Corollary}

\newtheorem{proposition}[theorem]{Proposition}

\newtheorem{definition}{Definition}[section]

\newtheorem{remark}{Remark}[section]
\newtheorem{example}{Example}[section]

\theoremstyle{definition}

\def\SS{\mathcal{S}}
\newcommand{\C}{\mathcal{C}}
\newcommand{\R}{\mathcal{R}}
\newcommand{\G}{\mathcal{G}}
\newcommand{\F}{\mathcal{F}}
\newcommand{\PP}{\mathcal{P}}

\newcommand{\RR}{\mathbb{R}}
\newcommand{\ZZ}{\mathbb{Z}}

\newcommand{\been}{\begin{enumerate}}
\newcommand{\enen}{\end{enumerate}}
\newcommand{\beit}{\begin{itemize}}
\newcommand{\enit}{\end{itemize}}

\newcommand{\specialcell}[2][c]{\begin{tabular}[#1]{@{}c@{}}#2\end{tabular}}

\def\ds{\displaystyle}
\def\one{\mathbbm{1}}

\DeclareMathOperator{\SSann}{span}

\DeclareMathOperator{\supp}{supp}

\providecommand{\abs}[1]{\lvert#1\rvert}
\providecommand{\norm}[1]{\lVert#1\rVert}

\title{Graphically balanced equilibria and stationary measures of reaction networks}
\author{Daniele Cappelletti\footnotemark[1] \and Badal Joshi\footnotemark[2]}

\begin{document}

 \footnotetext[1]{Department of Mathematics, University of Wisconsin-Madison.}
  \footnotetext[2]{Department of Mathematics, California State University San Marcos}

 \tikzset{every node/.style={auto}}
 \tikzset{every state/.style={rectangle, minimum size=0pt, draw=none, font=\normalsize}}
  \tikzset{bend angle=7}

 \maketitle

\begin{abstract}
The graph-related symmetries of a reaction network give rise to certain special equilibria (such as complex balanced equilibria) in deterministic models of dynamics of the reaction network. Correspondingly, in the stochastic setting, when modeled as a continuous-time Markov chain, these symmetries give rise to certain special stationary measures. Previous work by Anderson, Craciun and Kurtz identified stationary distributions of a complex balanced network; later Cappelletti and Wiuf developed the notion of complex balancing for stochastic systems. We define and establish the relations between reaction balanced measure, complex balanced measure, reaction vector balanced measure, and cycle balanced measure and prove that with mild additional hypotheses, the former two are stationary distributions. Furthermore, in spirit of earlier work by Joshi, we give sufficient conditions under which detailed balance of the stationary distribution of Markov chain models implies the existence of positive detailed balance equilibria for the related deterministic reaction network model. Finally, we provide a complete map of the implications between balancing properties of deterministic and corresponding stochastic reaction systems, such as complex balance, reaction balance, reaction vector balance and cycle balance. 

  
\end{abstract}

\section{Introduction}\label{sec:introduction}

%

Reaction networks are widely used mathematical models in biochemistry \cite{anderson:book, erdi:mathematical_models}. In these models, a finite collection of chemical species interact according to a finite set of possible chemical transformations.
An example of a reaction network is given by
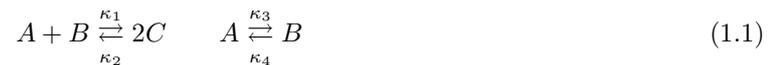
\begin{equation}\label{eq:reactions}
A+B \stackrel[\kappa_2]{\kappa_1}{\rightleftarrows} 2C \qquad A \stackrel[\kappa_4]{\kappa_3}{\rightleftarrows} B
  \end{equation}
In this case, $A$, $B$ and $C$ denote some distinct chemical species. Here, a molecule of $A$ and a molecule of $B$ can be turned into two molecules of $C$, and two molecules of $C$ can be turned back into one molecule of $A$ and one molecule of $B$. Additionally, a molecule of $A$ can be turned reversibly into a molecule of $B$. The numbers $\kappa_i$ denote the mass action reaction rate constants -- a larger value of a rate constant is associated with a higher propensity of the corresponding reaction. A more formal introduction is provided in Section \ref{sec:background}.

In the setting of biochemistry, different modeling regimes are considered. Specifically, if the counts of the molecules in the system of interest are low, then the evolution of these counts is modeled through a continuous time Markov chain. If more molecules are present, then the dynamical variables describe the concentrations of the different chemical species rather than their counts, and their time evolution is described by a system of stochastic differential equations \cite{kurtz1976}. Finally, if the number of molecules is so large that the random fluctuations in their counts can be safely ignored, then the time evolution of the species concentrations is modeled via a system of ordinary differential equations (ODEs). 

There is a rich history of relating the graphical properties of reaction networks with their dynamical features, especially for the ODE model \cite{feinberg1972, horn1972necessary, sturmfels, boros:deficiency, dickenstein2011far}. A theory that connects a graph to dynamics is naturally appealing, since graphical properties are simple to check, while the consequences for the dynamics are far-reaching.  
In this context, an important class of reaction network models is that of \emph{complex balanced systems}, whose positive equilibria satisfy a graphical balance condition (see Section \ref{sec:equilibria_deterministic}). It is known that if a positive complex balanced equilibrium exists, then all equilibria of the system are complex balanced and locally asymptotically stable. A special case of a complex balanced equilibrium is that of a \emph{detailed balanced} equilibrium, which is of particular importance in thermodynamics. 
A generalization of complex balanced equilibria has recently been studied in \cite{CFW:factors}.

As an example, we know from simply looking at the graphical properties of the chemical reaction network \eqref{eq:reactions}, that the deterministic model has positive detailed balanced equilibria for any choice of positive rate constants $\kappa_1$, $\kappa_2$, $\kappa_3$ and $\kappa_4$. We also know that the dynamics of the solution of the corresponding ODE system can neither be chaotic nor have oscillations, and will eventually converge to a locally stable positive equilibrium \cite{horn1972necessary, feinberg1972}.

Another topic of interest concerns the connection between dynamical properties of a continuous time Markov chain model and its corresponding ODE model. This study dates back to \cite{kurtz1972relationship}, where the ODE model is proven to be the weak limit of the continuous time Markov chain model, when the counts of the molecules is increased and properly rescaled. The interest in this topic is motivated by the desire to infer properties of the continuous time Markov chain model from the theory that has been developed for the deterministic model over the years. More recently, advances have been made in connecting complex balanced equilibria of the ODE models with the stationary distribution of the corresponding Markov chain model. Specifically, in \cite{anderson:product-form} it is shown that if an ODE model equipped with \emph{mass action} kinetics has a positive complex balanced equilibrium, then the Markov chain model has a product form Poisson stationary distribution (Theorem \ref{thm:anderson}). In \cite{ACK:explosion} it is further proven that such models are non-explosive. In \cite{joshi:detailed} it is proven that if an ODE model with mass action kinetics has a positive detailed balanced equilibrium, then the stationary distributions of the corresponding Markov chain model are detailed balanced, in the classical probabilistic sense \cite{norris:markov}. Finally, in \cite{cappelletti:complex_balanced} it is shown that if an ODE model with mass action kinetics is complex balanced, then the stochastic counterpart has a so called \emph{complex balanced stationary distribution}, and the converse holds as well. More properties of complex balanced stationary distributions are then studied in \cite{cappelletti:complex_balanced}, and a stochastic formulation of the Deficiency Zero Theorem is given.

In this paper, in order to fully compare the graphical properties of stochastic and deterministic models at equilibrium, we introduce {\em reaction vector balanced states} (see Definition \ref{def:balanced_states}), as the natural deterministic counterpart of detailed balanced distributions for stochastic models (called {\em reaction vector balanced distributions} here to make the connection more explicit). We study the properties of this new type of equilibrium, and how it relates to more classical notions of graphical balance.
Notably, unlike for complex balanced or detailed balanced states, the existence of a reaction vector balanced state does not imply existence or uniqueness of other equilibria within different positive compatibility classes (see Remark \ref{rem:rvb_not_uniqueness}).

We also extend the work of \cite{cappelletti:complex_balanced}, by studying graphically balanced measures, and not only balanced distributions. In particular, we define and establish the relations between reaction balanced measure, complex balanced measure and reaction vector balanced measure (see Definition \ref{def:balanced_measures}) and prove that with mild additional hypotheses, the former two are stationary distributions (see Theorem \ref{thm:exist_stat_dist}) while the same does not hold in general for the latter (see Remark \ref{rem:rvb_measure_not_distr}). 
  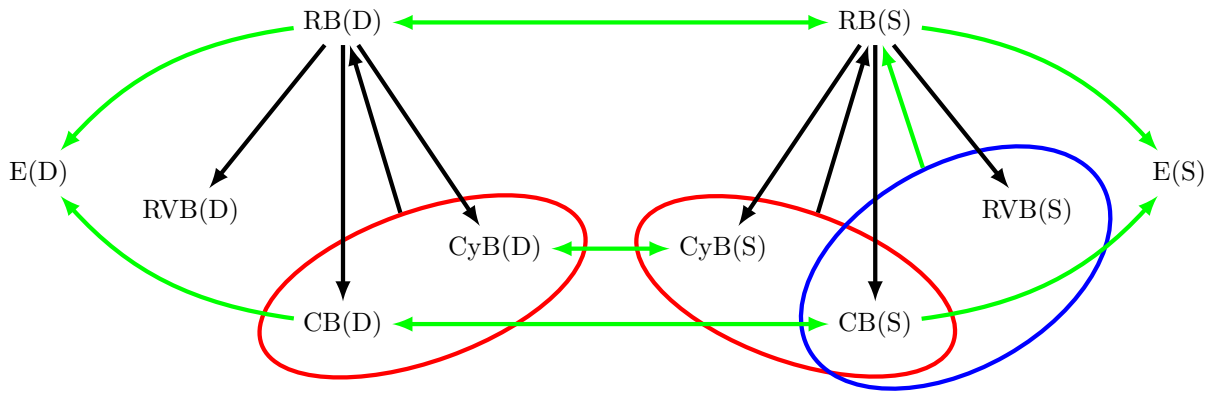
\begin{figure}[h!]
  \begin{center}
  \tikzset{>=latex}
  \tikzset{bend angle=20}
  \begin{tikzpicture}[baseline={(current bounding box.center)},scale=1]
   \node[state] (CBD)  at (-3.5,0)  {CB(D)};
   \node[state] (RBD)  at (-3.5,4)  {RB(D)};
   \node[state] (CyBD)  at (-1.5,1)  {CyB(D)};
   \node[state] (RVBD)  at (-5.5,1.5)  {RVB(D)};
    \node[state] (ED)  at (-7.5,2)  {E(D)};
   
   \node[state] (CBS)  at (3.5,0)  {CB(S)};
   \node[state] (RBS)  at (3.5,4)  {RB(S)};
   \node[state] (CyBS)  at (1.5,1)  {CyB(S)};
   \node[state] (RVBS)  at (5.5,1.5)  {RVB(S)};
   \node[state] (ES)  at (7.5,2)  {E(S)};
   
   \node [draw=red, ellipse, fit= (CBD) (CyBD), inner sep=-0.1cm, rotate=20, ultra thick] (CBDCyBD) {};
   
    \node [draw=red, ellipse, fit= (CBS) (CyBS), inner sep=-0.1cm, rotate=-20, ultra thick] (CBSCyBS) {};
    
    \node [draw=blue, ellipse, fit= (CBS) (RVBS), inner sep=-0.1cm, rotate=30, ultra thick] (CBSRVBS) {};
   
   \path
   (RBD) edge [->, ultra thick] (RVBD)
   (RBD) edge [->, ultra thick] (CBD)
   (RBD) edge [->, ultra thick] (CyBD)
   
    (RBS) edge [->, ultra thick] (RVBS)
   (RBS) edge [->, ultra thick] (CBS)
   (RBS) edge [->, ultra thick] (CyBS)
   
    (RBD) edge [<->, green, ultra thick] (RBS)
    (CyBD) edge [<->, green, ultra thick] (CyBS)
    (CBD) edge [<->, green, ultra thick] (CBS)
    
    (CBDCyBD) edge [->, ultra thick] (RBD)
    (CBSCyBS) edge [->, ultra thick] (RBS)
    (CBSRVBS) edge [->, green, ultra thick] (RBS)
    
    (RBD) edge [->, green, bend right, ultra thick] (ED)
    (CBD) edge [->, green, bend left, ultra thick] (ED)
    
    (RBS) edge [->, green, bend left, ultra thick] (ES)
    (CBS) edge [->, green, bend right, ultra thick] (ES);
\end{tikzpicture}
\caption{Notation key -- RB: Reaction Balance, CB: Complex Balance, RVB: Reaction Vector Balance, CyB: Cycle Balance. (D) denotes Deterministic reaction system and (S) denotes Stochastic reaction system. E(D): Deterministic reaction system has a positive equilibrium within each positive compatibility class. E(S): Stochastic reaction system has a stationary distribution within each irreducible component. An arrow represents implication. A solid black arrow represents implication for arbitrary kinetics and a green arrow for mass action kinetics only. An arrow from an ellipse containing multiple nodes indicates that both conditions must hold for the implication to hold. Absence of an arrow means that the corresponding implication does not exist in general either for arbitrary kinetics or for mass action kinetics.}
\label{summary_of_implications}
  \end{center}
 \end{figure}
 Further, in the spirit of \cite{anderson:product-form, joshi:detailed, cappelletti:complex_balanced}, we relate properties of the stationary measures of the Markov chain model with similar features of the equilibria of the corresponding ODE model. We prove that a complete symmetry between the results concerning graphical balancing equilibria and the results dealing with graphical balancing measures is lacking: specifically, in the stochastic sense two forms of graphical balance, namely reaction vector balance and complex balance, imply reaction balance (Theorem \ref{thm:rvb+cb=rb}). The same does not hold for the ODE model, as it is shown in Remark \ref{rem:no_implication_det}.


In the spirit of \cite{joshi:detailed}, we give sufficient conditions capable of ``lifting'' detailed balance of the Markov chain model to detailed balance of the ODE model. Specifically, in Corollary \ref{cor:cb_plus_detrb} we prove that if the Markov chain model has a detailed balanced distribution (here called `reaction vector balanced') which is complex balanced, then the associated ODE model admits a positive detailed balanced equilibrium (hence all equilibria are detailed balanced and locally asymptotically stable due to classical results of \cite{horn1972necessary}).

Figure \ref{summary_of_implications} presents a complete map of all connections between different kinds of balanced systems (see Definitions \ref{def:balanced_system} and \ref{def:stoch_balanced_system}), both in the deterministic and stochastic settings. The result that a mass action system is deterministically complex balanced if and only if it is stochastically complex balanced is owed to \cite{cappelletti:complex_balanced}. In \cite{horn1972necessary, feinberg1972}, it is shown that a deterministic complex balanced mass action system has a positive equilibrium within each positive compatibility class while in \cite{anderson:product-form} it is shown that the corresponding stochastic mass action system has a unique stationary distribution within each irreducible component.  The corresponding result that deterministic reaction balance in a mass action system implies a positive equilibrium in each positive compatibility class follows easily from the above. In \cite{dickenstein2011far}, it is shown that under general kinetics, a deterministic complex balanced system which is cycle balanced is reaction balanced. 
The other implications shown in the figure are either new results or generalizations of known results. Specifically, we introduce reaction balance, reaction vector balance, and cycle balance in the stochastic setting along with reaction vector balance in the deterministic setting. Furthermore, for some of the results (indicated by a solid black arrow in the figure), we consider arbitrary kinetics with the only restriction that the rate function be nonnegative-real valued. The positive orthant is not required to be invariant, and therefore in such cases we prove results for balanced states outside of $\RR^n_{> 0}$ or $\ZZ^n_{> 0}$. Finally, we emphasize that an absence of an arrow in Figure \ref{summary_of_implications} means that in general the implication does not exist. For instance, absence of a green arrow between RVB(D) and RVB(S) means that, even under the assumption of mass action kinetics, RVB(D) does not imply RVB(S) and vice versa (see Remark \ref{ex:srvb_drvb}). 
Figure \ref{summary_of_implications} is given as a quick reference, more detailed results can be found throughout the paper. In addition, a detailed summary of scheme of implications for the stochastic models is presented at the end of Section \ref{sec:expand_bridge}. 


The paper is organized as follows: in Section \ref{sec:background} the necessary definitions of reaction network theory are given, In Section \ref{sec:equilibria_deterministic}, equilibria of the ODE models with graphical balancing properties are discussed, and known results are presented. In Section \ref{sec:equilibria_stochastic}, invariant measures for the continuous time Markov chain are dealt with: graphical balance in this framework is discussed, and new properties are stated together with some known results. Sections \ref{sec:equilibria_deterministic} and \ref{sec:equilibria_stochastic} have a similar structure. In particular, in both sections we first deal with graphical balance under arbitrary kinetics, then specialize to results for mass action kinetics. In Section \ref{sec:bridge}, connections between equilibria of the ODE model and stationary distributions of the corresponding Markov chain model are proven, under the assumption of mass action kinetics. 

\section{Background}\label{sec:background}

 \subsection{Notation} 
 
Let $\RR$, $\RR_{\ge 0}$ and $\RR_{> 0}$ represent the reals, the nonnegative reals and the positive reals, respectively. Let $\ZZ$, $\ZZ_{\ge 0}$ and $\ZZ_{> 0}$ represent the integers, the nonnegative integers and the positive integers, respectively. 
For $v\in\RR^n$, $\norm{v}_1=\abs{v_1 + \ldots + v_n}$.
 For $v,w\in\RR^n$, $v \le w$ ($v < w$) means that $v_i \le w_i$ ($v_i < w_i$) for all $i \in \{1,\ldots, n\}$. 
For $v,w\in\RR^n$, we define  
 \begin{equation*}
 \one_{\{v \le w\}} = 
 \begin{cases}
 1 \quad , \quad v \le w \\
 0 \quad , \quad \mbox{ otherwise}.
 \end{cases}
 \end{equation*}
If $v>0$ then $v$ is said to be positive. 
If $x\in\RR^n_{\ge 0}$ and $v\in \ZZ^n_{\ge 0}$, we define
 $$x^v=\prod_{i=1}^n x_i^{v_i},\quad \text{and}\quad v!=\prod_{i=1}^n v_i!,$$
 with the conventions that $0!=1$ and $0^0=1$. 

 \subsection{Reaction networks}  
 
 A reaction network is a triple $\G=(\SS,\C,\R)$, where $\SS=\{X_1,X_2,\dots,X_n\}$ is a set of $n$ species, $\C$ is a set of $m$ complexes, and $\R\subseteq \C\times\C$ is a set of $r$ reactions, such that  $(y,y)\notin\R$ for any  $y\in\C$. The complexes are  linear combinations of species over $\ZZ_{\ge 0}$, identified with vectors in $\ZZ_{\ge 0}^n$ (which can be therefore embedded in $\RR^n$).  A reaction $(y,y')\in\R$ is denoted by $y\to y'$, and the vector $y'-y$ is the corresponding \emph{reaction vector}.  We refer to $y$ as the {\em reactant complex} of the reaction $y \to y'$ and to $y'$ as the {\em product complex}. We require that every  species has a nonzero coordinate in at least one complex and that every complex is either a reactant complex or a product complex of at least one reaction. With this convention, there are no ``redundant'' species or complexes and $\G$ is uniquely determined by $\R$. In  \eqref{eq:reactions}, there are $n=3$ species ($A,B,C$), $m=2$ complexes ($A+B, 2B$), and $r=4$ reactions. 
 
 For convenience, we allow for a set of reactions to be empty. In such a case, $n=m=r=0$ and the associated reaction network is called the {\em empty network}.
 
Suppose that $\G=(\SS,\C,\R)$ is a reaction network and $\R' \subseteq \R$. Then $\F = (\SS |_{\R'}, \C |_{\R'}, \R')$ is said to be a {\em subnetwork} of $\G$. Since a reaction network is determined by its set of reactions, we will equivalently say that $\R'$ determines a subnetwork of $\R$ (or of $\G$) without any loss of clarity. 

A reaction network $\G=(\SS,\C,\R)$ can be determined by the directed graph with node set $\C$ and edge set $\R$ in a natural manner. This graph is known as \emph{reaction graph} and we will frequently refer to it in the rest of the paper.

A reaction network $\G$ is \emph{weakly reversible} if every reaction $y\to y'\in\R$ is contained in a closed directed path of the reaction graph. Recall that a closed directed path is defined as a path $y_1\to y_2\to\dots\to y_k$ in the reaction graph such that $k\geq2$ and $y_1=y_k$.
Moreover, $\G$ is \emph{reversible} if for any reaction $y\to y'\in\R$, $y'\to y$ is in $\R$. It is clear that each reversible reaction network is also weakly reversible, since every reaction $y\to y'$ is contained in the cycle $y\to y'\to y$. As an example, the network in \eqref{eq:reactions} is reversible, and therefore weakly reversible. By definition, the empty network is reversible.
 
  The \emph{stoichiometric subspace} of $\G$ is the linear subspace of $\RR^n$ generated by the reaction vectors, namely
 $$S=\SSann(y'-y\vert y\to y'\in\R).$$
 For $v\in\RR^n$, the sets $(v+S)\cap\RR^n_{\ge 0}$ are called the \emph{stoichiometric compatibility classes} of $\G$. 
 

 \subsection{Reaction systems}
 We will consider dynamics of a reaction network with $n$ species both on $\RR^n$ and $\ZZ^n$. $\RR^n$ is the usual underlying state space for deterministic models involving ordinary differential equations, while the state space is $\ZZ^n$ for stochastic models involving continuous-time Markov chains. We do not consider stochastic differential equations (ordinary differential equations with a noise term) in this paper, but in passing we mention that this is an instance where a stochastic model has the underlying state space $\RR^n$ (see for example \cite{kurtz1976}, and \cite{ruth, bibbona:strong, bibbona:weak} for more recent developments on the subject).
  
\begin{definition}
Let $\G$ be a reaction network. Suppose that to each reaction $y \to y' \in \R$ there is associated  a nonnegative-valued {\em rate function} $\lambda_{y \to y'}$, whose domain is the state space ($\lambda_{y \to y'}: \RR^n \to \RR_{\ge 0}$ or $\lambda_{y \to y'}: \ZZ^n \to \RR_{\ge 0}$). By kinetics on $\G$, we mean the correspondence between reactions and rate functions. 
\begin{equation*}
\Lambda: (y \to y') \mapsto \lambda_{y \to y'}
\end{equation*}
The pair $(\G,\Lambda)$ will be referred to as a {\em reaction system}. 
\end{definition}

The range of $\lambda_{y \to y'}$ besides being nonnegative real-valued usually has some additional restrictions, which prevent the trajectory from exiting the nonnegative orthant. However, we will not make these restrictions because they are not necessary for the results in this article. In particular, the results stated in this paper for arbitrary kinetics do not require the dynamical system (either stochastic or deterministic) to be restricted to $\RR^n_{\ge 0}$.

It is natural to think of the kinetic system $(\G,\Lambda)$ as a labelled graph, where $\G$ is the underlying graph and the rate function $\lambda_{y \to y'}$ labels each edge $y \to y' \in \R$.

We further give the following definition:
\begin{definition}\label{def:activity}
 Let $(\G,\Lambda)$ be a reaction system and let  $\Gamma$ be a subset of the state space, i.e. $\Gamma \subseteq \ZZ^n$  or $\Gamma \subseteq \RR^n$. Then,
 \begin{itemize}
  \item a reaction $y \to y' \in \R$ is {\em active in $\Gamma$} if $\lambda_{y \to y'}(x) > 0$ for some  $x\in\Gamma$ (if a reaction $y \to y' \in \R$ is active in a singleton set $\Gamma = \{x\}$, we say that $y \to y' \in \R$ is active at $x$);
  \item the {\em active subnetwork} of $(\G,\Lambda)$ in $\Gamma$ (denoted by $\G_\Gamma$) is the network determined by the set of reactions in $\R$ that are active in $\Gamma$ (if $\Gamma = \{x\}$ is a singleton set, we will denote $\G_{\{x\}}$ by $\G_x$ for simplicity, and will refer to it as the active subnetwork at $x$);
  \item the set $\Gamma$ is {\em active} if $\G_\Gamma=\G$, that is if for every reaction $y\to y'\in\R$ there exists $x\in\Gamma$ such that $\lambda_{y\to y'}(x)>0$ (if $\Gamma = \{x\}$ is a singleton set, we will simply say that $x$ is active).
 \end{itemize}
 \end{definition}
 

 
 \subsubsection{Deterministic/continuous dynamics}
 
Let $(\G,\Lambda)$ be a deterministic reaction system. In the deterministic case, the evolution of the species concentrations $z(t) \in \RR^n$ is determined by the following system of ODEs
 \begin{equation*}
  \frac{d z}{dt} =\sum_{y\to y'\in\R}(y'-y)\lambda_{y\to y'}(z) 
 \end{equation*}
The evolution of the deterministic reaction system is confined to translations of the stoichiometric subspace, that is for any $t\geq 0$
 \begin{equation*}
 z(t)\in (z(0)+S). 
 \end{equation*}
If, as usual, the dynamics are restricted to the positive orthant, in accordance with the chemistry interpretation of the model, then the solution $z(t)$ is confined to its stoichiometric compatibility class, 
 \begin{equation*}
 z(t)\in (z(0)+S)\cap\RR^n_{\ge 0} =: \PP_{z(0)}, 
 \end{equation*}
 where $\PP_z$ denotes the compatibility class containing $z \in \RR^n$. We say that a compatibility class $\PP_{z}$ is {\em positive} if $\PP_z \cap\RR^n_{> 0}$ is nonempty. 
$c \in \RR^n$ is said to be an {\em equilibrium} of the deterministic reaction system $(\G,\Lambda)$ if 
\begin{equation*}
\sum_{y\to y'\in\R}(y'-y)\lambda_{y\to y'}(c) = 0.
\end{equation*}
 If $c \in \RR^n_{>0}$ ($c \in \RR^n_{\ge 0}$), we say that $c$ is positive (nonnegative) and denote it by $c > 0$ ($c \ge 0$). 

 \subsubsection{Stochastic/discrete dynamics}

Let $(\G,\Lambda)$ be a stochastic reaction system. In the discrete setting, the underlying state space is $\ZZ^n$. A vector $x = (x_1 ,\ldots, x_n) \in \ZZ^n$ represents the population $x_i$ of each species $i = 1, \ldots, n$. When a reaction $y \to y'$ that is active at $x$ occurs,  the population changes from $x$ to $x + y' - y$. 



We say that $x' \in \ZZ^n$ is {\em accessible} from $x \in \ZZ^n$, if either $x' = x$ or there is a sequence of states $(x=u_0, u_1, \ldots, u_n = x')$ such that for each consecutive pair of states $(u_i, u_{i+1})$ ($0 \le i \le n-1$), there is an active reaction $y \to y' \in \R$ at $u_i$ with $y' - y = u_{i+1} - u_i$.  A non-empty set $\Gamma\subseteq\ZZ^{n}$ is an \emph{irreducible component} of $(\G,\Lambda)$ if for all $x\in\Gamma$ and all $u\in\ZZ^{n}$, $u$ is accessible from $x$ if and only if $u\in\Gamma$.  According to Definition \ref{def:activity}, an irreducible component $\Gamma$ is \emph{active} if for all reactions $y\to y'\in\R$ there exists a state $x\in\Gamma$ such that $y\to y'$ is active at $x$. Note that active irreducible components are called `positive' in \cite{cappelletti:complex_balanced}, where they were first introduced. 

%
\begin{definition} 
Let $(\G,\Lambda)$ be a stochastic reaction system and let $\pi$ be a measure on $\ZZ^n$. 
$\pi$ is said to be a {\em stationary measure} of $(\G, \Lambda)$ if the following holds for all $x \in \ZZ^n$: 
\begin{equation*}
 \pi(x) \sum_{y\to y' \in \R}  \lambda_{y \to y'}(x) = \sum_{y\to y' \in \R} \pi(x+y - y') \lambda_{y \to y'}(x+y - y'),
\end{equation*}
\end{definition}

\begin{definition}
\been
\item $\pi$ is said to be a {\em $\sigma$-finite measure} if $\pi(x) < \infty$ for all $x \in \ZZ^n$. $\pi$ is said to be a {\em finite measure} if $\pi\left(\ZZ^n\right) < \infty$. $\pi$ is said to be a {\em distribution} if $\pi\left(\ZZ^n\right) =1$.
\item We define {\em support} of a measure $\pi$ on $\ZZ^n$, denoted by $\supp(\pi)$, to be the smallest set $T$ such that $\pi(\ZZ^n \setminus T)=0$.
\item If $\Omega \subseteq \ZZ^n$ is such that $\supp(\pi) \cap \Omega \ne \emptyset$, we say that $\pi$ is {\em non-null} on $\Omega$.
\item If $\Omega \subseteq \ZZ^n$ is such that $\supp(\pi)$ is nonempty and is contained in $\Omega$, then we say that $\pi$ {\em exists within $\Omega$}.  
\enen
\end{definition}

We note here that if $\pi$ is a non-null stationary measure of $(\G,\Lambda)$, then $\pi$ must exist within the union of all irreducible components of $(\G,\Lambda)$ (see for example \cite{norris:markov}).

In the usual setting of chemical reaction networks, the dynamics of the Markov chain are restricted to $\ZZ^n_{\geq0}$. In this case, only irreducible components contained in $\ZZ^n_{\geq0}$ are considered and the stationary distributions exist within $\ZZ^n_{\geq0}$. However, these restrictions are not necessary for the results on arbitrary kinetics that we present in this paper.

\subsubsection{Mass action kinetics}
An important choice of kinetics, both deterministic and stochastic, is mass action kinetics. 
The deterministic ODE mass action model has been shown to arise as a certain large population limit of the stochastic mass action model \cite{kurtz1972relationship}. For the purposes of this article, we will assume the two mass action models as given without explicitly realizing one model as a limit of the other. 
We define the two models below. 
\begin{definition} Consider a reaction network $\G = (\SS,\C,\R)$. 
\been
\item By deterministic mass action kinetics, we mean the correspondence $K_D: (y \to y') \mapsto \lambda_{y \to y'}$ where
\begin{equation*}
\lambda_{y\to y'}(z)=\mathbbm{1}_{\{z\geq 0\}}\kappa_{y\to y'} z^{y} \mbox{ for some } \kappa_{y\to y'} > 0\text{ and for all }z\in\RR^n
\end{equation*}
The pair $(\G,K_D)$ is called a \emph{deterministic mass action system}. 
\item By stochastic mass action kinetics, we mean the correspondence $K_S: (y \to y') \mapsto \lambda_{y \to y'}$ where
\begin{equation*}
\lambda_{y\to y'}(x)=\kappa_{y\to y'}\frac{x!}{(x-y)!}\mathbbm{1}_{\{x\geq y\}} \mbox{ for some } \kappa_{y\to y'} > 0\text{ and for all }x\in\ZZ^n.
\end{equation*}
The pair $(\G, K_S)$ is called a \emph{stochastic mass action system}. 
\enen
In both cases, the constants $\kappa_{y\to y'}$ are referred to as \emph{rate constants}.
\end{definition}  
We will often need to refer to a deterministic mass action system $(\G,K_D)$ and its {\em corresponding} stochastic mass action system $(\G,K_S)$ within the same context. The correspondence is via the reaction network $\G$, which is the same, and by making the same choice of rate constants. 
\begin{remark}\label{rem:mass_action_no_indicator}
 An equivalent formulation of the stochastic mass action rates restricted to $\ZZ^n_{\geq 0}$ is the following:
 $$\lambda_{y\to y'}(x)=\kappa_{y\to y'}\frac{x!}{(x-y)!}\mathbbm{1}_{\{x\geq y\}}=\kappa_{y\to y'}\prod_{i=1}^nx_i(x_i-1)\cdots(x_i-y_i+1)$$
 for all $x\in\ZZ^n_{\geq 0}$. Note that no indicator function appears on the right-hand side of the equation. The rates of stochastic mass action kinetics are multivariate polynomial functions, restricted to nonnegative integer arguments. Moreover, the polynomials $\prod_{i=1}^nx_i(x_i-1)\cdots(x_i-y_i+1)$ for different $y$ have different leading terms $x^y$ (by ordering the terms first according to their degree, then according to lexicographic order), and therefore are linearly independent on $\RR$.
\end{remark}

\section{Graph-related equilibria in deterministic reaction systems}\label{sec:equilibria_deterministic}

%

In the deterministic setting, graph-related symmetries of a reaction network result in certain special states, which we now define. 

 \begin{definition} \label{def:balanced_states}
 Consider a deterministic reaction system $(\G, \Lambda)$, let $c\in\RR^n$, and with a slight abuse of notation define $\lambda_{y \to y'} = 0$ if $y \to y' \notin \R$. Then
 \begin{enumerate}[(a)]
  \item $c$ is said to be a \emph{reaction balanced state of $(\G, \Lambda)$} (or {\em detailed balanced state of $(\G, \Lambda)$}) if 
  for every pair of complexes $y,y'\in\C$,
 \begin{equation}\label{eq:detbal}
 \lambda_{y\to y'}(c)=\lambda_{y'\to y}(c)
 \end{equation}
   \item $c$ is said to be a \emph{complex balanced state of $(\G, \Lambda)$} if for every complex $y\in\C$, 
  \begin{equation} \label{eq:compbal}
  \sum_{y'\in\C}\lambda_{y\to y'}(c)=\sum_{y'\in\C} \lambda_{y'\to y}(c)
   \end{equation}
 \item $c$ is said to be a {\em reaction vector balanced state of $(\G, \Lambda)$} if for all $\xi \in \RR^n$, 
  \begin{equation} \label{eq:rvb}
 \sum_{y\to y'\in\R \,:\, y' -y = \xi} \lambda_{y \to y'}(c) = \sum_{y\to y'\in\R \,:\, y' -y = -\xi} \lambda_{y \to y'}(c)
 \end{equation}
 \item $c$ is said to be a \emph{cycle balanced state of $(\G, \Lambda)$} if for every sequence of distinct complexes $(y_1, \ldots, y_j) \subseteq \C$ where $j \ge 3$, 
  \begin{equation} \label{def:cyc_bal}
 \prod_{i=1}^j \lambda_{y_i \to y_{i+1}}(c)=\prod_{i=1}^j \lambda_{y_{i+1} \to y_i}(c)
   \end{equation}
where by definition $y_{j+1} := y_1$. 
 \end{enumerate} 
 \end{definition}
 
 Part (a) of Definition \ref{def:balanced_states} appears as `detailed balanced state' in the chemical reaction network theory literature. For the purposes of this paper, we refer to the concept as `reaction balanced state' instead, in order to avoid any potential confusion with the concept of detailed balance of Markov chain theory literature. 
 
 Part (c) of Definition \ref{def:balanced_states} does not appear in the literature on deterministic reaction networks to the best of our knowledge. One reason may be that the existence of reaction vector balanced states does not relate well with structural properties of the reaction graph (see Theorem \ref{thm:necessary}) and does not imply existence and uniqueness of other equilibria within different positive compatibility classes (see Remark \ref{rem:rvb_not_uniqueness}). However, we introduce reaction vector balanced state because it is a natural deterministic counterpart of the concept of `detailed balanced measure', a basic notion in stochastic theory. For the purposes of this paper, we refer to detailed balanced measure as reaction vector balanced measure, which makes the correspondence with the deterministic setting explicit, and avoids potential confusion with detailed balance of deterministic chemical reaction networks.
 
Part (d) of Definition \ref{def:balanced_states} appears as `formal balance' in \cite{dickenstein2011far}, which is an extension to general kinetics of the cycle conditions appearing in \cite{feinberg1989necessary} and \cite{schuster1989generalization}, both of which in turn are extensions of Wegscheider's cycle conditions valid for monomolecular reactions. 
 
\begin{example}
Consider the following deterministic reaction system $(\G,\Lambda)$. 
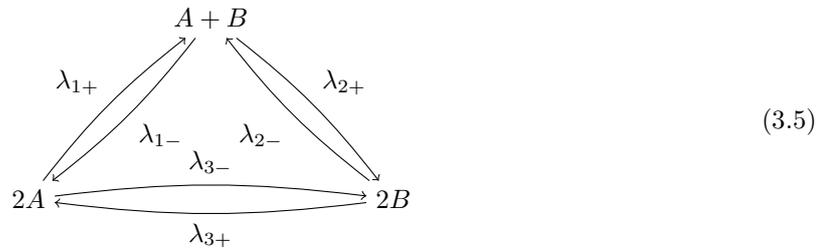
\begin{equation}\label{eq:example_triangle}
  \begin{tikzpicture}[baseline={(current bounding box.center)}, scale=0.8]
   \node[state] (A)  at (0,0)  {$2A$};
   \node[state] (B)  at (3,3)  {$A+B$};
   \node[state] (C)  at (6,0)  {$2B$};
   \path[->]
    (A) edge[bend left] node {$\lambda_{1+}$} (B)
    	edge[bend left] node {$\lambda_{3-}$} (C)
    (B) edge[bend left] node {$\lambda_{2+}$} (C)
         edge[bend left] node {$\lambda_{1-}$} (A)
    (C) edge[bend left] node {$\lambda_{3+}$} (A)
          edge[bend left] node {$\lambda_{2-}$} (B);
  \end{tikzpicture}
 \end{equation}
 For purposes of this example, $\lambda_{*}$ that is written on top of an arrow denotes the rate function (and not the mass action rate constant) of the corresponding reaction. Let $c$ be a vector of $\RR^n$.
\been[(a)]
\item $c$ is a reaction balanced state of $(\G,\Lambda)$ if the following hold:
\been[(i)]
\item $\lambda_{1+}(c) = \lambda_{1-}(c)$, 
\item $\lambda_{2+}(c) = \lambda_{2-}(c)$,
\item $\lambda_{3+}(c) = \lambda_{3-}(c)$. 
\enen

\item $c$ is a complex balanced state of $(\G,\Lambda)$ if the following hold:
\been[(i)]
\item $\lambda_{2+}(c) +  \lambda_{1-}(c) = \lambda_{2-}(c) + \lambda_{1+}(c)$, 
\item $\lambda_{3+}(c) + \lambda_{2-}(c) = \lambda_{3-}(c) + \lambda_{2+}(c)$,
\item $\lambda_{1+}(c) + \lambda_{3-}(c) = \lambda_{1-}(c) + \lambda_{3+}(c)$. 
\enen

\item $c$ is a reaction vector balanced state of $(\G,\Lambda)$ if the following hold:
\been[(i)]
\item $\lambda_{1+}(c) + \lambda_{2+}(c) = \lambda_{1-}(c) + \lambda_{2-}(c)$, 
\item $\lambda_{3+}(c) = \lambda_{3-}(c)$.
\enen

\item $c$ is a cycle balanced state of $(\G,\Lambda)$ if the following holds: \\
$
\lambda_{1+}(c) \lambda_{2+}(c) \lambda_{3+}(c) = \lambda_{1-}(c) \lambda_{2-}(c) \lambda_{3-}(c).
$

\enen
 \end{example}
 
 We now define what it means for a deterministic reaction system to be `graphically balanced', i.e. as being one of complex balanced, reaction balanced, reaction vector balanced or cycle balanced. Recall that $c$ is said to be active if $\lambda_{y \to y'}(c)>0$ for all $y \to y' \in \R$. 
\begin{definition}\label{def:balanced_system}
Let $(\G,\Lambda)$ be a deterministic reaction system. Suppose that $(\G,\Lambda)$ has at least one active equilibrium and suppose that every active equilibrium is complex balanced (or reaction balanced, or reaction vector balanced or cycle balanced, resp.). Then we say that $(\G,\Lambda)$ is a {\em complex balanced (or reaction balanced, or reaction vector balanced or cycle balanced, resp.)} reaction system. 
\end{definition}

\begin{remark}\label{rem:balanced_system_mass_action}
 For a deterministic mass action system $(\G,K_D)$, we have that $(\G,K_D)$ is complex balanced (or reaction balanced, or reaction vector balanced or cycle balanced, resp.) if there is at least one positive equilibrium, and all positive equilibria are complex balanced (or reaction balanced, or reaction vector balanced or cycle balanced, resp.).
This follows from the fact that for deterministic mass action, $\lambda_{y\to y'}(c)>0$ if and only if the entries of $c$ relative to the species appearing in $y$ are positive. Moreover, if $c$ is an equilibrium and $\lambda_{y\to y'}(c)>0$ for all $y\to y'\in\R$, then necessarily all species appearing in a product complex need to appear in a source complex as well (or at $c$ they would not be at equilibrium). Since we require that all species of $\G$ appear in at least one complex, it follows that an equilibrium $c$ with $\lambda_{y\to y'}(c)>0$ for all $y\to y'\in\R$ is necessarily positive.
\end{remark}

\subsection{Results for arbitrary kinetics}

\begin{theorem}[Balanced states are equilibria] \label{thm:bal_eq}
Let $(\G,\Lambda)$ be a deterministic reaction system and let $c \in \RR^n$. Suppose that any one of the following conditions holds:
\been[(i)]
\item $c$ is a reaction balanced state of $(\G,\Lambda)$. 
\item $c$ is a complex balanced state of $(\G,\Lambda)$. 
\item $c$ is a reaction vector balanced state of $(\G,\Lambda)$. 
\enen
Then  $c$ is an equilibrium of $(\G, \Lambda)$. 
\end{theorem}
\begin{proof}
By summing \eqref{eq:detbal} over $y' \in \C$, we see that a reaction balanced state is a complex balanced state. To see that a complex balanced state is an equilibrium, multiply both sides of \eqref{eq:compbal} by $y$ and then sum over $y \in \C$. 
If $c$ is a reaction vector balanced state, multiply both sides of \eqref{eq:rvb} by $\xi$ and then sum over the vectors $\xi\in\RR^n$ such that at least one among $\xi$ and $-\xi$ is a reaction vector.
\end{proof}

\begin{remark} \label{rem:cyc_bal}
A cycle balanced state is not necessarily an equilibrium. In fact, for a deterministic mass action system, either every state is cycle balanced or no positive state is cycle balanced: this can be shown by proving that if a positive state $c$ is cycle balanced, then necessarily all states must be cycle balanced. Assume that $c$ is positive and cycle balanced. From the form of deterministic mass action kinetics and from \eqref{def:cyc_bal}, it follows that for every sequence of distinct complexes $(y_1, \ldots, y_j) \subseteq \C$ where $j \ge 3$, 
  \begin{equation*}
 c^{\sum_{i=1}^j y_i}\prod_{i=1}^j \kappa_{y_i \to y_{i+1}}=c^{\sum_{i=1}^j y_{i+1}}\prod_{i=1}^j \kappa_{y_{i+1} \to y_i},
   \end{equation*}
where $y_{j+1}= y_1$ and $\kappa_{y\to y'}$ is considered to be 0 if $y\to y'\notin\R$. Note that the powers of the positive vector $c$ on the two sides of the equation are the same, hence we have
  \begin{equation*}
 \prod_{i=1}^j\kappa_{y_i \to y_{i+1}}=\prod_{i=1}^j \kappa_{y_{i+1} \to y_i}.
   \end{equation*}
It follows that for any $z\in\RR^n$ we have
  \begin{align*}
  \mathbbm{1}_{\{z\geq 0\}}z^{\sum_{i=1}^j y_i}\prod_{i=1}^j\kappa_{y_i \to y_{i+1}}=\mathbbm{1}_{\{z\geq 0\}}z^{\sum_{i=1}^j y_{i+1}}\prod_{i=1}^j \kappa_{y_{i+1} \to y_i},
   \end{align*}
which implies that $z$ is cycle balanced.
\end{remark}

The following is a generalization of well-known results (e.g.\ \cite{horn1972necessary}). The generalization covers the new type of balanced states described in this paper. Furthermore the results extend to the arbitrary kinetics considered here and also to the possibility for the states to be outside of the positive orthant, $\RR^n_{>0}$. 

\begin{theorem}[Necessary conditions for existence of a balanced state] \label{thm:necessary}
 Let $(\G,\Lambda)$ be a deterministic reaction system and let $c \in \RR^n$. Let $\G_c=(\SS_c, \C_c, \R_c)$ be the active subnetwork of $(\G,\Lambda)$ at $c$. Then, the following holds:
 \been[(i)]
 \item If $c$ is a reaction balanced state of $(\G,\Lambda)$, then $\G_c$ is reversible. 
  \item If $c$ is a complex balanced state of $(\G,\Lambda)$, then $\G_c$ is weakly reversible. 
   \item If $c$ is a reaction vector balanced state of $(\G,\Lambda)$, then for every $y \to y' \in \R_c$, there exists $\tilde y \to \tilde y' \in \R_c$ with $y + \tilde y  = y' + \tilde y'$. 
   \item Let $c$ be a cycle balanced state of $(\G,\Lambda)$. If for a sequence of distinct complexes $(y_1, \ldots, y_j)$ where $j \ge 3$, 
$\prod_{i=1}^j \lambda_{y_i \to y_{i+1}}(c) > 0$ with $y_{j+1}=y_1$, then $y_i \to y_{i+1}$ is a reversible reaction for all $i \in \{ 1, \ldots, j\}$. 
 \enen
 
 \end{theorem}
 \begin{proof}
 The proof of parts (i), (iii) and (iv)  follows immediately from Definition \ref{def:balanced_states} by noticing that for the balancing conditions to hold, if one side of the equations is nonzero, then so must be the other side. In case $\R_c$ is empty, recall that the empty network is reversible.

 The statement of part (ii) is a slight generalization of Theorem 2B of \cite{horn1972necessary}, where the result was first proven under mild conditions on the kinetics. The proof of the original result of \cite{horn1972necessary} is valid under the more general setting of this paper as well. However, a shorter way to prove the theorem consists in a trick similar to one used in \cite{dickenstein2011far}.
 
 By definition, the reactions of $\R_c$ are the reactions that are active at $c$. If $\R_c$ is empty, then $\G_c$ is weakly reversible and the proof is completed. Otherwise, if $c$ satisfies \eqref{eq:compbal}, then the network $\G_c$ endowed with mass action kinetics with $\kappa_{y\to y'}=\lambda_{y\to y'}(c)$ for all $y\to y'$ is complex balanced, a complex balanced state being the vector whose entries are all 1. Hence, $\G_c$ is weakly reversible by Theorem 2B of \cite{horn1972necessary}.
 \end{proof}
   
It is well known that a positive reaction balanced state is complex balanced \cite{horn1972necessary}. Moreover, it is known that a strictly positive state $c$ is complex balanced and cycle balanced if and only if $c$ is reaction balanced, under natural mild conditions on the kinetics \cite{dickenstein2011far}. Here we give a generalization of the result that covers more general kinetics (according to the definition of this paper) and includes the notion of reaction vector balance. 

\begin{theorem}[Relations between different balanced states] \label{thm:det_rel}
Let $(\G,\Lambda)$ be a deterministic reaction system. Let $c \in \RR^n$. 
\been[(i)]
\item If $c$ is reaction balanced then it is reaction vector balanced, complex balanced and cycle balanced. 
\item If $c$ is complex balanced and cycle balanced, then $c$ is reaction balanced. 
\enen
\end{theorem} 
 \begin{proof}
 Let $c$ be a reaction balanced state of $(\G,\Lambda)$. Then it follows, by summing \eqref{eq:detbal} over $y' \in \C$ that $c$ is a complex balanced state. Furthermore, by summing \eqref{eq:detbal} over $y \to y' \in \R : y' -y = \xi$ it follows that $c$ is a reaction vector balanced state.  For every sequence $(y_1, \ldots, y_j)$, reaction balance of $c$ implies that $\lambda_{y_i \to y_{i+1}}(c) = \lambda_{y_{i+1} \to y_{i}}(c)$. Cycle balance of $c$ then follows from taking the product over $i \in \{1, \ldots, j\}$ on both sides of the identity. For part (ii), if $\lambda_{y\to y'}(c)=0$ for all $y\to y'\in \R$ then clearly $c$ is reaction balanced. Otherwise, consider the active subnetwork at $c$, denoted by $\G_c=(\SS_c, \C_c, \R_c)$, and let $n_c$ be the cardinality of $\SS_c$. Equip $\G_c$ with mass action kinetics, with $\kappa_{y\to y'}=\lambda_{y\to y'}(c)$ for all $y\to y'\in \R_c$. Then, the vector $\tilde{c}\in\RR^{n_c}$ whose entries are all 1 is complex balanced and cycle balanced for $\G_c$ with this choice of mass action kinetics. We can then conclude by \cite[Theorem 1.1]{dickenstein2011far} that $\tilde{c}$ is reaction balanced. This means that for any $y\to y'\in\R_c$ we have
 $$\lambda_{y\to y'}(c)=\lambda_{y\to y'}(c)\tilde{c}^y=\lambda_{y'\to y}(c)\tilde{c}^y=\lambda_{y'\to y}(c),$$
 which concludes the proof.
\end{proof}

\begin{remark}\label{rem:no_implication_det}
In general, no one individual condition of the three: complex balance, reaction vector balance, and cycle balance implies any of the remaining two, not even in the more restricting setting of mass action kinetics. We check the remaining pairwise conditions below. Note that this is a novel study since the concept of reaction vector balance is not present in the literature of deterministic reaction networks, to the best of our knowledge.
\been 
\item (Complex balance \& Reaction vector balance $\centernot \implies$ Cycle balance) 
Consider the mass action system $(\G,K_D)$ shown below, where the chosen rate constants are written on top of the arrow of the corresponding reaction:
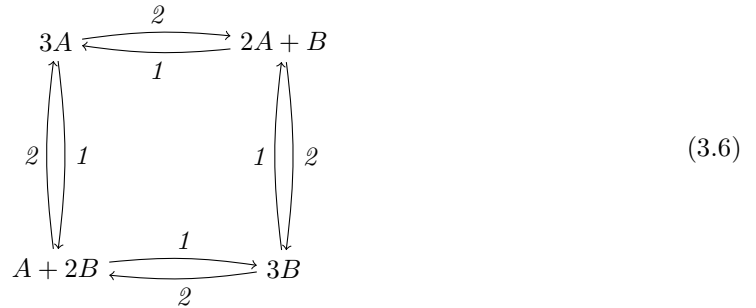
\begin{equation} \label{eq:square}
  \begin{tikzpicture}[baseline={(current bounding box.center)}]
   \node[state] (A)  at (0,3)  {$3A$};
   \node[state] (B)  at (3,3)  {$2A+B$};
   \node[state] (C)  at (3,0)  {$3B$};
   \node[state] (D)  at (0,0)  {$A+2B$};
   \path[->]
    (A) edge[bend left] node {2} (B)
    	edge[bend left] node {1} (D)
    (B) edge[bend left] node {2} (C)
    	edge[bend left] node {1} (A)
    (C) edge[bend left] node {2} (D)
    	edge[bend left] node {1} (B)
    (D) edge[bend left] node {2} (A)
    	edge[bend left] node {1} (C);
  \end{tikzpicture}
 \end{equation}
The state $(1,1)$ is complex balanced and reaction vector balanced. However, $(1,1)$ is neither reaction balanced nor cycle balanced. In fact, all positive equilibria are of the form $(s,s)$ for some $s>0$ and each of these equilibria is both complex balanced and reaction vector balanced. This shows that $(\G,K_D)$ is complex balanced as well as reaction vector balanced (see Definition \ref{def:balanced_system} and Remark \ref{rem:balanced_system_mass_action}) but $(\G,K_D)$ is neither cycle balanced nor reaction balanced. 

\item (Reaction vector balance \& Cycle balance $\centernot \implies$ Complex balance) 
%
Consider the reaction system in \eqref{eq:example_triangle}, endowed with with mass action kinetics with the following rate constants:
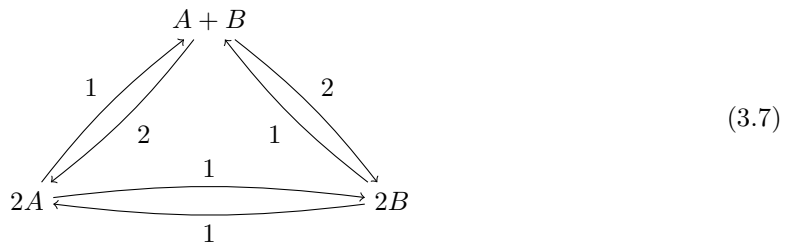
\begin{equation}
  \begin{tikzpicture}[baseline={(current bounding box.center)},scale=0.8]
   \node[state] (A)  at (0,0)  {$2A$};
   \node[state] (B)  at (3,3)  {$A+B$};
   \node[state] (C)  at (6,0)  {$2B$};
   \path[->]
    (A) edge[bend left] node {$1$} (B)
    	edge[bend left] node {$1$} (C)
    (B) edge[bend left] node {$2$} (C)
         edge[bend left] node {$2$} (A)
    (C) edge[bend left] node {$1$} (A)
          edge[bend left] node {$1$} (B);
  \end{tikzpicture}
 \end{equation}
 Then, $(1,1)$ is a positive equilibrium of the mass action system $(\G,K_D)$ that is both reaction vector balanced and cycle balanced, but is not complex balanced.  Moreover, all equilibria are of the form $(s,s)$ for some $s\geq0$, and they are all reaction vector balanced and cycle balanced. It follows that $(\G,K_D)$ is reaction vector balanced and cycle balanced (see Definition \ref{def:balanced_system} and Remark \ref{rem:balanced_system_mass_action}) but it is not complex balanced.
 \enen
\end{remark}



\subsection{Results for mass action kinetics}
The following result on complex balancing is known
\cite{horn1972necessary, feinberg1972, cappelletti:complex_balanced} and we extend the result to reaction balancing.
\begin{theorem}[Existence and uniqueness of positive equilibria] \label{thm:bal}
Let $(\G,K_D)$ be a deterministic mass action system. If $(\G,K_D)$ has a positive complex balanced (reaction balanced, resp.) state, then there exists a unique positive equilibrium within every positive compatibility class of $(\G,K_D)$, and every equilibrium of $(\G,K_D)$ is complex balanced (reaction balanced, resp.). 
\end{theorem}
\begin{proof}  Suppose that $(\G,K_D)$ has a positive complex balanced state. It is known that there exists a unique positive equilibrium within every positive compatibility class of $(\G,K_D)$, and every positive equilibrium of $(\G,K_D)$ is complex balanced \cite{horn1972necessary, feinberg1972}, i.e. that $(\G, K_D)$ is complex balanced. Furthermore, in Theorem 4 of \cite{cappelletti:complex_balanced}, it is shown that if $(\G, K_D)$ is complex balanced then all equilibria (not just the positive ones) of $(\G,K_D)$ are complex balanced. 

Let $c$ be a positive reaction balanced state of $(\G,K_D)$. By Theorem \ref{thm:det_rel}, $c$ is cycle balanced and complex balanced. By Remark \ref{rem:cyc_bal}, every state in $\RR^n_{\ge 0}$ is cycle balanced. By the result on complex balance, there is a unique positive equilibrium within every positive compatibility class of $(\G,K_D)$ and every one of these positive equilibria is complex balanced. Since these equilibria are both complex balanced and cycle balanced, by part (ii) of Theorem \ref{thm:det_rel}, these equilibria are reaction balanced. This shows that $(\G,K_D)$ is reaction balanced. Suppose now that $c$ is a boundary equilibrium, i.e. $c \in \RR^n_{\ge 0} \setminus \RR^n_{> 0}$ of a reaction balanced system $(\G,K_D)$. If no reaction is active at $c$, then $c$ is clearly reaction balanced. So we may assume that $\G_c$ is not the empty network. Then $c$ is a complex balanced equilibrium of $\G_c = (\SS_c, \C_c, \R_c)$ by the result on complex balance, in particular $\G_c$ is weakly reversible. Since $\G$ is cycle balanced and $\G_c$ is a subnetwork of $\G$, $\G_c$ is cycle balanced. This follows from the fact that cycles in $\G_c$ are also cycles in $\G$. Let $\tilde c$ denote the restriction of $c$ to the species in $\SS_c$. Note that $\tilde c > 0$ because $\G_c$ is weakly reversible and active at $c$. Thus $\tilde c$ is a positive complex balanced equilibrium of $\G_c$ which is cycle balanced. It follows that $\tilde c$ is a reaction balanced equilibrium of $\G_c$, by part (ii) of Theorem \ref{thm:det_rel}. Therefore, $c$ is a reaction balanced equilibrium of $\G$. 
\end{proof}

Note that uniqueness of positive equilibria within positive compatibility classes is not guaranteed for complex balanced systems with arbitrary kinetics (see, for instance, \cite{muller2012generalized}). 
\begin{remark}
Since the positive compatibility classes of $(\G,K_D)$ partition $\RR^n_{> 0}$, Theorem \ref{thm:bal} implies that if a mass action system $(\G,K_D)$ possesses a positive complex balanced (resp. reaction balanced) state, then every positive equilibrium of $(\G,K_D)$ is complex balanced (resp. reaction balanced). Therefore, a mass action system $(\G,K_D)$ with a positive complex balanced (resp. reaction balanced) state is complex balanced (resp. reaction balanced). Moreover, due to Remarks \ref{rem:cyc_bal} and \ref{rem:balanced_system_mass_action}, a mass action system $(\G,K_D)$ possessing a positive cycle balanced equilibrium is cycle balanced.
\end{remark}
\begin{remark}\label{rem:rvb_not_uniqueness}
We introduced in this paper the concept of reaction vector balanced state, and it is natural to wonder whether for this kind of equilibium something similar to Theorem \ref{thm:bal} holds. While a reaction vector balanced state is an equilibrium, unfortunately there is no statement about uniqueness corresponding to Theorem \ref{thm:bal}. To see this consider the following mass action system $(\G,K_D)$. 
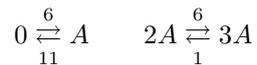
\begin{equation*}
0 \stackrel[11]{6}{\rightleftarrows} A \qquad 2A \stackrel[1]{6}{\rightleftarrows} 3A
\end{equation*}
$(\G,K_D)$ has three distinct positive reaction vector balanced equilibria $z=1$, $z=2$, and $z=3$ within the same positive compatibility class. Furthermore, even when a positive reaction vector balanced state exists, there may be positive compatibility classes which do not contain any positive equilibria. To see this, consider the following mass action system.
\begin{equation}\label{eq:ACR}
B \stackrel[]{1}{\to}  A \qquad A+B \stackrel[]{1}{\to} 2B
\end{equation}
There exists a positive equilibrium $(1,l-1)$ within every compatibility class with $z_A+z_B = l >1$, and each of these equilibria is reaction vector balanced. However, there is no positive equilibrium in the positive compatibility classes with $z_A+z_B=l \le 1$. 

\end{remark}

\section{Graph-related stationarity in stochastic reaction systems}\label{sec:equilibria_stochastic}
In the stochastic setting, graph-related symmetries of a reaction network result in certain special measures, which we now define. 
\begin{definition} \label{def:balanced_measures}
Consider a stochastic reaction system $(\G,\Lambda)$, let $\pi$ be a measure defined on $\ZZ^n$ and with a slight abuse of notation define $\lambda_{y \to y'} = 0$ if $y \to y' \notin \R$. Then 
\been[(a)]
\item $\pi$ is said to be a {\em reaction balanced measure} if for every pair of complexes $y, y' \in \C$ and every $x \in \ZZ^n$
\begin{equation} \label{eq:rxnbal}
\pi(x) \lambda_{y \to y'}(x) = \pi(x+y'-y) \lambda_{y' \to y}(x+y'-y). 
\end{equation}
\item $\pi$ is said to be a {\em complex balanced measure} if for every complex $y \in \C$ and every $x \in \ZZ^n$
\begin{equation} \label{eq:compbal_m}
\pi(x) \sum_{y' \in \C}  \lambda_{y \to y'}(x) = \sum_{y' \in \C} \pi(x+y'-y) \lambda_{y' \to y}(x+y'-y). 
\end{equation} 
\item $\pi$ is said to be a {\em reaction vector balanced measure} if for every every $x \in \ZZ^n$ and every $\xi \in \ZZ^n$
\begin{equation} \label{eq:rxnvecbal}
\pi(x) \sum_{y\to y'\in\R \,:\, y' -y = \xi} \lambda_{y \to y'}(x) = \pi(x+\xi)\sum_{y\to y'\in\R \,:\, y' -y = -\xi} \lambda_{y \to y'}(x+\xi)
\end{equation}
\item $\pi$ is said to be a \emph{cycle balanced measure} if for every every $x \in \ZZ^n$ and every sequence of distinct complexes $(y_1, \ldots, y_j) \subseteq \C$ where $j \ge 3$, 
  \begin{equation} \label{eq:cyc_bal_m}
 \prod_{i=1}^j \pi(x+y_i)\lambda_{y_i \to y_{i+1}}(x+y_i)=\prod_{i=1}^j\pi(x+y_{i+1}) \lambda_{y_{i+1} \to y_i}(x+y_{i+1}). 
   \end{equation}
 \enen
\end{definition}

The definition of a complex balanced distribution is owed to \cite{cappelletti:complex_balanced}, here we extend the definition by considering more general measures. Note that the terms  involving $\pi$ in \eqref{eq:cyc_bal_m} can be canceled if they are positive, hence the property of cycle balance is more related to $\supp(\pi)$ rather than $\pi$ itself. We summarize the different types of balanced equilibria and balanced measures in Table \ref{table:defs_bal}. 


 \begin{table}[h!] 
  \centering
   \begin{adjustbox}{max width=\textwidth}
 \begin{tabular}{|c||c|c|c|}
 \hline
 ~ & Deterministic setting & Stochastic setting & Holds $\forall$ \\
 \hline
 \hline
 \specialcell{Reaction \\ balance} & $\ds \lambda_{y \to y'}(c) = \lambda_{y' \to y}(c)$ \tablefootnote{Detailed Balance of Chemical Reaction Network Theory}  & $\ds \pi(x) \lambda_{y \to y'}(x) = \pi(x+y'-y) \lambda_{y' \to y}(x+y'-y)$ & $x, y, y'$ \\
 \hline
  \specialcell{Reaction \\ vector \\ balance} & $\ds  \sum_{y\to y'\in\R \,:\, y' -y = \xi} \lambda_{y \to y'}(c) = \sum_{y\to y'\in\R \,:\, y' -y = -\xi} \lambda_{y \to y'}(c)$  & $\ds \pi(x) \sum_{y\to y'\in\R \,:\, y' -y = \xi} \lambda_{y \to y'}(x) = \pi(x+\xi)\sum_{y\to y'\in\R \,:\, y' -y = -\xi} \lambda_{y \to y'}(x+\xi)$ \tablefootnote{Detailed Balance of Markov Chain Theory} & $x, \xi$ \\
 \hline
  \specialcell{Complex \\ balance} & $\ds \sum_{y' \in \C}\lambda_{y \to y'}(c) = \sum_{y' \in \C}\lambda_{y' \to y}(c)$  & $\ds \pi(x) \sum_{y' \in \C}  \lambda_{y \to y'}(x) = \sum_{y' \in \C} \pi(x+y'-y) \lambda_{y' \to y}(x+y'-y)$ & $x, y$ \\
  \hline
   \specialcell{Cycle \\ balance} & $\ds   \prod_{i=1}^j \lambda_{y_i \to y_{i+1}}(c)=\prod_{i=1}^j \lambda_{y_{i+1} \to y_i}(c)$  & $\ds  \prod_{i=1}^j \pi(x+y_i)\lambda_{y_i \to y_{i+1}}(x+ y_i) = \prod_{i=1}^j \pi(x+y_{i+1})\lambda_{y_{i+1} \to y_i}(x+y_{i+1})
$ & \specialcell{$x , (y_1, \ldots, y_j)$ with \\ $y_i$ distinct and $j\geq3$} \\
  \hline
  \specialcell{Stationary \\ measure} & ~  & $\ds \pi(x) \sum_{y\to y' \in \R}  \lambda_{y \to y'}(x) = \sum_{y\to y' \in \R} \pi(x+y - y') \lambda_{y \to y'}(x+y - y')$ & $x$ \\
  \hline
  \specialcell{Equilibrium} &
  $\ds \sum_{y \to y' \in \R} y \lambda_{y \to y'} (c)  = \sum_{y \to y' \in \R} y' \lambda_{y \to y'} (c) $ & ~  & ~\\
 \hline
 \end{tabular}
 \end{adjustbox}
 \caption{Summary of definitions of various balanced equilibria in deterministic setting and various balanced measures in corresponding stochastic setting.} \label{table:defs_bal}
 \end{table}


Note that part (a) of Definition \ref{def:balanced_measures} is the natural stochastic analog of part (a) of Definition \ref{def:balanced_states}, which is usually called detailed balance in reaction network theory. However, we could not refer to part (a) of Definition \ref{def:balanced_measures} as detailed balance, since ``detailed balance measure" is a reserved name in general Markov chain theory. In the case of reaction networks, the usual detailed balance of Markov chain theory coincides with reaction vector balance defined in part (c) of Definition \ref{def:balanced_measures}.

Note that if $x$ is an absorbing state, meaning that $\lambda_{y\to y'}(x)=0$ for all $y\to y'\in \R$, then any measure with support on $\{x\}$ is reaction balanced, reaction vector balanced, complex balanced and cycle balanced. Indeed, both sides of the equations in Definition \ref{def:balanced_measures} are zero.

We now define what it means for a stochastic reaction system to be graphically balanced, in the sense that the system is complex balanced, reaction balanced, reaction vector balanced or cycle balanced. 
\begin{definition} \label{def:stoch_balanced_system}
Let $(\G,\Lambda)$ be a stochastic reaction system. Suppose that $(\G,\Lambda)$ has at least one stationary distribution within an active irreducible component and every stationary distribution of $(\G,\Lambda)$ within an active irreducible component is complex balanced (or reaction balanced, or reaction vector balanced, or cycle balanced, resp.). Then we say that $(\G,\Lambda)$ is a {\em complex balanced (or reaction balanced, or reaction vector balanced, or cycle balanced, resp.)} reaction system. 
\end{definition}

\subsection{Results for arbitrary kinetics}

\begin{theorem}[Balanced measures are stationary] \label{thm:stoch_imp}
Let $(\G,\Lambda)$ be a stochastic reaction system. Suppose that $\pi$ is a measure that satisfies at least one of the following:
\been[(i)]
\item $\pi$ is a reaction balanced measure of $(\G,\Lambda)$.
\item $\pi$ is a complex balanced measure of $(\G,\Lambda)$.
\item $\pi$ is a reaction vector balanced measure of $(\G,\Lambda)$.
\enen
Then $\pi$ is a stationary measure of $(\G,\Lambda)$.
\end{theorem}
\begin{proof}
To see that a reaction balanced measure is stationary, sum \eqref{eq:rxnbal} over $y\to y' \in \R$. We get that a complex balanced measure is stationary by summing \eqref{eq:compbal_m} over $y \in \C$. To see that a reaction vector balanced measure is stationary, sum \eqref{eq:rxnvecbal} over all reaction vectors $\xi$.
\end{proof}

There is no result corresponding to the following in the deterministic setting. We say that a reaction system $(\G,\Lambda)$ is {\em non-explosive} if for every initial distribution, the resulting stochastic process is not explosive (in the sense of \cite{norris:markov}). 
\begin{proposition}\label{prop:nonexplosive}
 Let $(\G, \Lambda)$ be a stochastic reaction system with at least one active irreducible component. If there exists a cycle balanced stationary distribution $\pi$ with $\supp(\pi)=\ZZ^n$, then $(\G, \Lambda)$ is cycle balanced. Moreover, if $(\G,\Lambda)$ is non-explosive  and if there exists a complex balanced (or reaction balanced, or reaction vector balanced, resp.) distribution $\pi$ with $\supp(\pi)=\ZZ^n$, then $(\G, \Lambda)$ is complex balanced (or reaction balanced, or reaction vector balanced, resp.).
\end{proposition}
\begin{proof}
 If $\pi$ is cycle balanced and $\supp(\pi)=\ZZ^n$, we can cancel $\pi$ on the two sides of  \eqref{eq:cyc_bal_m} and we have that for every $x \in \ZZ^n$ and for every sequence of distinct complexes $(y_1, \ldots, y_j) \subseteq \C$ of $\G$ where $j \ge 3$
  \begin{equation*}
 \prod_{i=1}^j \lambda_{y_i \to y_{i+1}}(x+ y_i) = \prod_{i=1}^j \lambda_{y_{i+1} \to y_i}(x+y_{i+1}),
   \end{equation*}
 where as usual $y_{j+1} := y_1$ and $\lambda_{y \to y'} = 0$ if $y \to y' \notin \R$. Necessarily, every $\sigma$-finite measure (whether or not it is a stationary distribution) must be cycle balanced.
Finally, if $(\G, \Lambda)$ is non-explosive, then within every irreducible component the only stationary distribution is proportional to $\pi$ \cite{norris:markov}. The conclusion follows by noting that equations \eqref{eq:rxnbal}, \eqref{eq:compbal_m} and \eqref{eq:rxnvecbal} still hold after multiplication by a constant.
\end{proof}
\begin{remark}\label{rem:nonexplosive}
 The above proposition may also be stated by only assuming that $\supp(\pi)$ is the union of irreducible components, and at least one active irreducible component exists. Moreover, if the dynamics are restricted to $\ZZ^n_{\geq0}$ and all states outside the nonnegative orthant are considered absorbing (as for stochastic mass action kinetics), then the result can be even restricted to $\supp(\pi)\subseteq \ZZ^n_{\geq0}$. It is worth noting here that complex balanced (and therefore reaction balanced) stochastic mass action systems are necessarily non-explosive \cite{ACK:explosion}. 
\end{remark}

The following result is analogous to Theorem \ref{thm:det_rel} for graphically balanced states. The proof of part \eqref{part4} is similar to the one for the 
the deterministic setting presented in \cite{JM:balance}. 

\begin{theorem}[Relations between different balanced measures] \label{thm:stoch_rel} 
Let $(\G,\Lambda)$ be a stochastic reaction system. 
\been[(i)]
\item A reaction balanced measure of $(\G,\Lambda)$ is reaction vector balanced, complex balanced and cycle balanced. 
\item \label{part4} If $\pi$ is a complex balanced and cycle balanced measure of $(\G,\Lambda)$, then $\pi$ is reaction balanced. 
\enen
\end{theorem} 
\begin{proof}
Suppose that $\pi$ is a reaction balanced measure of $(\G,\Lambda)$. Then, summing \eqref{eq:rxnbal} over $y' \in \C$ gives \eqref{eq:compbal_m}, and summing \eqref{eq:rxnbal} over the reaction vectors gives \eqref{eq:rxnvecbal}. For every sequence $(y_1, \ldots, y_j)$, reaction balance of $\pi$ implies that $\pi(x+y_i)\lambda_{y_i \to y_{i+1}}(x+y_i) = \pi(x+y_{i+1})\lambda_{y_{i+1} \to y_{i}}(x+y_{i+1})$. Cycle balance of $\pi$ then follows from taking the product over $i \in \{1, \ldots, j\}$ on both sides of the identity.

To prove part \eqref{part4}, suppose that $\pi$ is a complex balanced but not reaction balanced measure of $(\G,\Lambda)$. We will show that $\pi$ is not cycle balanced. We use the convention that if $y \to y' \notin \R$, then $\lambda_{y \to y'}(x) = 0$ for all $x \in \ZZ^n$. Define the flux at $x \in \ZZ^n$ from the complex $y' \in \C$ to the complex $y \in \C$ to be 
\begin{equation*}
\rho_{y, y'}(x) := \pi(x+y'-y)\lambda_{y' \to y}(x + y'-y) - \pi(x) \lambda_{y \to y'}(x).
\end{equation*} 
Since $\pi$ is not reaction balanced, there exists an $x \in \ZZ^n$ and a pair of complexes $y_1, y_2 \in \C$ such that $\rho_{y_1, y_2}(x) \ne 0$. Clearly, $\rho_{y_1, y_2}(x) > 0$ if and only if $\rho_{y_2,y_1}(x+y_2-y_1)<0$. So we assume without loss of generality that there exists an $x \in \ZZ^n$ and a pair of complexes $y_1, y_2 \in \C$ such that $\rho_{y_1, y_2}(x) > 0$. Since $\pi$ is complex balanced (but not reaction balanced), there exists a complex $y_3 (\ne y_1) \in \C$ such that $\rho_{y_2,y_3}(x + y_2 - y_1) > 0$.  Continuing this argument, there exists a sequence of complexes $(y_1, y_2, y_3, \ldots)$ such that $\rho_{y_i,y_{i+1}}(x + y_i - y_1) > 0$ for all $i \ge 1$. However, since there are only finitely many complexes in a reaction network, eventually we get a nontrivial cycle, i.e. a state $x \in \ZZ^n$ and a sequence of distinct complexes  $\{y_1, y_2, \ldots, y_j\} \subseteq \C$ such that $\rho_{y_i, y_{i+1}}(x+ y_i - y_1) > 0$ where $i \in \{1,\ldots, j\}$ and $y_{j+1} = y_1$. This implies that for $i \in \{1,\ldots, j\}$, $\pi(x+y_{i+1} - y_1) \lambda_{y_{i+1} \to y_i} (x+y_{i+1} - y_1) > \pi(x+y_{i} - y_1) \lambda_{y_{i} \to y_{i+1}} (x+y_{i} - y_1)$. Taking product over $i \in \{1,\ldots, j\}$, we get that $\prod_{i=1}^j \pi(x+y_{i+1} - y_1)\lambda_{y_{i+1} \to y_i} (x+y_{i+1} - y_1) > \prod_{i=1}^j  \pi(x+y_i - y_1)\lambda_{y_{i} \to y_{i+1}} (x+y_{i} - y_1)$, which implies that $\pi$ is not cycle balanced, thus completing the proof of the claim. 
\end{proof}

\begin{remark} As in the deterministic setting (see Remark \ref{rem:no_implication_det}), reaction vector balance and cycle balance do not imply complex balance for a stochastic reaction system. To see this, consider the mass action system $(\G,K_S)$ depicted below:
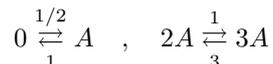
\begin{equation*}
0 \stackrel[1]{1/2}{\rightleftarrows} A \quad, \quad 2A \stackrel[3]{1}{\rightleftarrows} 3A
\end{equation*}
Cycle balance is trivially satisfied, since there are no cycles involving three or more non-repeating complexes. 
The resulting stochastic process is a birth and death process with a unique stationary distribution within $\ZZ_{\ge 0}$, which is reaction vector balanced \cite{norris:markov}. The negative states, if considered, are all absorbing states. Hence, $(\G,K_S)$ is cycle balanced and reaction vector balanced. We only need to check that for $(\G,K_S)$, the stationary distribution within $\ZZ_{\ge 0}$ is not reaction balanced. In fact, the only measure on $\ZZ_{\ge 0}$ balancing the first pair of reversible reactions is $\pi(n) = (1/2)^n/n!$, which does not balance the second pair of reversible reactions. This shows that $(\G,K_S)$ is not reaction balanced.  

Surprisingly enough, and in contrast to deterministic mass action systems, in the setting of stochastic mass action systems, complex balance and reaction vector balance imply reaction balance  (see Theorem \ref{thm:rvb+cb=rb}).
\end{remark}

The following result extends part of \cite[Theorem 18]{cappelletti:complex_balanced} by only requiring measures in the hypotheses instead of distributions, and also extends to reaction balance and reaction vector balance as well. Moreover, the definition of kinetics we have in this paper is more general than that commonly used in the context of stochastic reaction networks. Furthermore, in \cite{cappelletti:complex_balanced}, it is assumed that $\lambda_{y\to y'}(x)>0$ if and only if $x\geq y$, which we do not require here.

\begin{theorem}[Necessary conditions for existence of a balanced measure] \label{thm:nec_conds_measure}
 Let $(\G,\Lambda)$ be a stochastic reaction system and let $\pi$ be a measure within an irreducible component $\Gamma$ of $(\G,\Lambda)$. Let $\G_\Gamma=(\SS_\Gamma, \C_\Gamma, \R_\Gamma)$ be the active subnetwork of $(\G,\Lambda)$ in $\Gamma$.
 \been[(i)]
 \item If $\pi$ is a reaction balanced measure of $(\G,\Lambda)$, then $\G_\Gamma$ is reversible. 
  \item If $\pi$ is a complex balanced measure of $(\G,\Lambda)$, then $\G_\Gamma$ is weakly reversible. 
   \item If $\pi$ is a reaction vector balanced measure of $(\G,\Lambda)$, then for every $y \to y' \in \R_\Gamma$, there exists $\widetilde y \to \widetilde y' \in \R_{\Gamma}$ with $y + \widetilde y  = y' + \widetilde y'$. 
 \enen
  \end{theorem}
\begin{proof}
If $\pi$ is either reaction balanced, complex balanced or reaction vector balanced, then by Theorem \ref{thm:stoch_imp}, $\pi$ is a stationary measure. It follows from the standard theory of Markov chains that $\pi$ is positive on every state of $\Gamma$. 
Consider $y \to y' \in \R_\Gamma$, and let $x \in \Gamma$ be such that $\lambda_{y \to y'}(x) >0$. If $\pi$ is reaction balanced, then from \eqref{eq:rxnbal}, it follows that $\lambda_{y' \to y}(x+y'-y) >0$, so that $y' \to y \in \R_\Gamma$. If $\pi$ is reaction vector balanced, then from \eqref{eq:rxnvecbal}, it follows that there is a $\widetilde y \to \widetilde y' \in \R_\Gamma$ such that $\widetilde y - \widetilde y' = y' - y$. 
Finally, if $\pi$ is complex balanced, we consider the continuous time Markov chain with state space $\C_\Gamma$ and transition rate from complex $y'$ to $y''$ to be $q(y',y'') :=\lambda_{y'\to y''}(x-y+y')$, where as usual the expression is $0$ if $y'\to y''\notin\R_\Gamma$. Since $\pi$ is complex balanced, $\nu(y''):=\pi(x-y+y'')$ is a stationary distribution for the contructed Markov chain. Since $\nu(y)=\pi(x)>0$, $y$ must be recurrent, which implies that $y\to y'$ is necessarily contained in a closed directed path $y\to y'\to \cdots \to y$, which concludes the proof. 
\end{proof}

\subsection{Results for mass action kinetics}

The main result in this section is the stochastic analog of the classical Horn, Jackson and Feinberg theory for deterministic mass action reaction systems. We start with a consequence of \cite[Corollary 19]{cappelletti:complex_balanced}. The original result \cite[Corollary 19]{cappelletti:complex_balanced} is restated in Section \ref{sec:bridge} of this paper as Theorem \ref{thm:complex-balance}.
\begin{theorem}\label{thm:one_to_many_cb_dist}
 Let $(\G,K_S)$ be a stochastic mass action system. If $(\G,K_S)$ possesses a complex balanced distribution within some active irreducible component of $(\G,K_S)$, then there exists a unique stationary distribution within every irreducible component of $(\G,K_S)$, and every stationary distribution of $(\G,K_S)$ is complex balanced.
\end{theorem}

In the following result, we extend Theorem \ref{thm:one_to_many_cb_dist} by only requiring the hypothesis of a complex balanced measure instead of a complex balanced distribution. Moreover, we also consider reaction balanced measures, in analogy with the deterministic statement in Theorem \ref{thm:bal}.
\begin{theorem}[Existence and uniqueness of stationary distribution] \label{thm:stoch_bal}
Let $(\G,K_S)$ be a stochastic mass action system. If $(\G,K_S)$ possesses a $\sigma$-finite complex balanced (reaction balanced, resp.) measure within some active irreducible component of $(\G,K_S)$, then there exists a unique stationary distribution within every irreducible component of $(\G,K_S)$, and every stationary distribution of $(\G,K_S)$ is complex balanced (reaction balanced, resp.). 
\end{theorem}
Theorem \ref{thm:stoch_bal} will be proved after stating and proving Proposition \ref{thm:exist_stat_dist}. 

\begin{remark}
Theorem \ref{thm:stoch_bal} implies that a stochastic mass action system which possesses a $\sigma$-finite, complex balanced (reaction balanced, resp.) measure within some active irreducible component is a complex balanced (reaction balanced, resp.) reaction system. 
\end{remark}

\begin{remark}\label{rem:no_one_to_many}
In general, a statement analogous to Theorem \ref{thm:stoch_bal} does not hold when reaction balance or complex balance is replaced by reaction vector balance. 
\been
\item Consider the mass action system
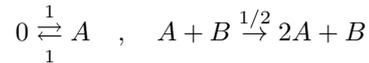
\begin{equation*}
0 \stackrel[1]{1}{\rightleftarrows} A \quad, \quad A+B \stackrel[]{1/2}{\to} 2A+B
\end{equation*}
The active irreducible components are given by $x_B=l$ where $l\in\ZZ_{\geq0}$. Within every such component, the associated process is a birth and death process with a unique stationary measure (up to multiplication by contants). Moreover, the process is positive recurrent only for $l=1$ \cite{norris:markov}. It follows that a reaction vector balanced stationary distribution within an active irreducible component exists, but no stationary distribution exists for $l\geq 2$.
\item Even in cases where a stationary distribution exists within every irreducible component, the stationary distribution may be reaction vector balanced in some irreducible components and not in others. Consider
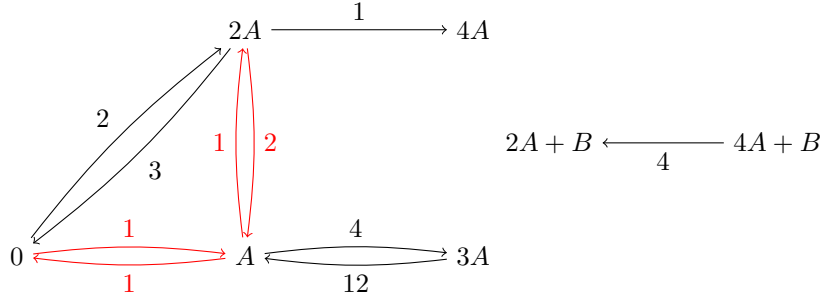
\begin{equation*}
  \begin{tikzpicture}[baseline={(current bounding box.center)}]
   \node[state] (1)  at (0,0)  {$0$};
   \node[state] (2)  at (3,0)  {$A$};
   \node[state] (3)  at (6,0)  {$3A$};
    \node[state] (4)  at (3,3)  {$2A$};
     \node[state] (5)  at (6,3)  {$4A$};
      \node[state] (6)  at (7,1.5)  {$2A+B$};
     \node[state] (7)  at (10,1.5)  {$4A+B$};
   \path[->]
    (1) edge[red,bend left] node {$1$} (2)
    	edge[bend left] node {$2$} (4)
    (2) edge[red,bend left] node {$1$} (1)
         edge[bend left] node {$4$} (3)
         edge[red,bend left] node {$1$} (4)
    (3) edge[bend left] node {$12$} (2)
    (4) edge[red,bend left] node {$2$} (2)
          edge[bend left] node {$3$} (1)
    (4) edge[] node {$1$} (5)
    (7) edge[] node {$4$} (6);
  \end{tikzpicture}
 \end{equation*}
The irreducible components correspond to $x_B=l$ with $l\in\ZZ_{\geq0}$ is a nonnegative integer. It can be checked that a stationary distribution exists within every irreducible component. The reactions corresponding to the edges in red have reaction vector $\pm(1,0)$ while the reactions corresponding to the edges in black have reaction vector $\pm(2,0)$. We denote by $q_l(i,j)$ the transition rate from the state $(i,l)$ to $(j,l)$. We have
$$\begin{array}{ll}
 q_l(i,i+1)=i+1 & q_l(i,i+2)=(i+1)(i+2)\\
 q_l(i,i-1)=i(2i-1) & q_l(i,i-2)=i(i-1)[4i^2l+2i(6-10l)+24l-21]\\
\end{array}$$ 
 and $q_l(i,j)=0$ otherwise. If a reaction vector balanced measure $\pi$ exists within an irreducible component, then it must be a stationary measure for the subnetwork determined by the red reactions. The only possibility is that
 \begin{equation*}
 \pi(i,l) = M_l\prod_{j=0}^{i-1}\frac{1}{2j+1},
 \end{equation*}
 where $M_l$ is a normalization constant depending on the irreducible component $x_B=l$. It can be checked that $\pi$ is a stationary measure only for $l=1$. It follows that the stationary distribution within $x_B=1$ is reaction vector balanced, and the stationary distributions within other irreducible components are not.
\enen
\end{remark}

\begin{proposition}[Existence of stationary distribution] \label{thm:exist_stat_dist}
Let $(\G,K_S)$ be a stochastic mass action system and let $\pi$ be a $\sigma$-finite complex balanced measure within an irreducible component of $(\G,K_S)$. Then $\pi$ is a finite measure (and so is a stationary distribution, up to a normalization constant). 
\end{proposition}

\begin{proof}
Let $\mu$ be a $\sigma$-finite complex balanced measure within the irreducible component $\Gamma$ of the mass action system $(\G,K_S)$. By Theorem \ref{thm:stoch_imp}, $\mu$ is a stationary measure, and therefore from standard theory of continuous time Markov chains, $\mu$ is positive on all states of $\Gamma$. 

If $\Gamma$ is a singleton containing an absorbing state $x\notin\ZZ^n_{\geq0}$, then the result follows. So for the rest of the proof we focus on the case $\Gamma\subseteq\ZZ^n_{\geq0}$.

Now, let $\G_\Gamma=(\SS_\Gamma, \C_\Gamma, \R_\Gamma)$ denote the active subnetwork of $(\G, K_S)$ in $\Gamma$. By Theorem \ref{thm:nec_conds_measure}, the network $\G_\Gamma$ is weakly reversible. Denote by $(\G_\Gamma, K_S)$ the mass action system whose rate constants are naturally determined by those of $(\G, K_S)$, and let $n_\Gamma$ be the cardinality of $\SS_\Gamma$.

If $n_\Gamma<n$, then necessarily the species in $\SS\setminus\SS_\Gamma$ have constant counts on $\Gamma$: indeed if the counts of a species $S$ are not constant on $\Gamma$, then there must be a reaction $y\to y'\in\R_\Gamma$ with $y'-y$ having a non-zero entry for $S$, which implies $S\in\SS_\Gamma$. Hence, if $n_\Gamma<n$, after potentially reordering the species, $\Gamma=\Gamma'\times\{q\}$ with $\Gamma'\subseteq \ZZ^{n_d}$ being an active irreducible component of $\G_\Gamma$ and $q\in\ZZ^{n-n_d}$. Finally, we have that the measure on $\Gamma'$ defined by $\mu'(x')=\mu(x',q)$ is a complex balanced measure of $(\G_\Gamma,K_S)$. Note that the latter holds even if $n_\Gamma=n$, in which case $\mu'=\mu$.

Now introduce a set of fictitious species $\{S_y\}_{y\in\C_\Gamma}$, one fictitious species for one complex of $\G_\Gamma$, and consider the mass action system, denoted by $(\widehat \G, K_S)$ and described by the set of reactions
 \begin{equation*}
 \widehat{\R} := \{ y+S_y\to y'+S_{y'} : y\to y'\in\R_\Gamma \} ,
 \end{equation*}
The resulting reaction network $\widehat \G$ is clearly weakly reversible since the reactions $\widehat \R$ are in one-to-one correspondence with the reactions $\R_\Gamma$. By using this correspondence, we equip $\widehat \G$ with mass action kinetics, with the same rate constants as $(\G_\Gamma, K_S)$ (that correspond to a subset of those of $(\G, K_S)$). Furthermore, $\widehat{\G}=(\widehat{\SS},\widehat{\C}, \widehat{\R})$ has \emph{deficiency} zero. (See the proof of \cite[Theorem 18]{cappelletti:complex_balanced}. In the same paper, a definition of deficiency is given, together with a discussion on its meaning.) Since the reaction network $\widehat{\G}$ has deficiency zero and is weakly reversible, from \cite{feinberg1972}, it follows that the network $(\widehat{\G},K_D)$ is complex balanced for every choice of positive reaction rate constants. 

Let $m_\Gamma$ be the number of complexes of $\G_\Gamma$ (equivalently of $\widehat{\G}$) and let $(x',\widehat{x}) \in\ZZ_{\ge 0}^{n_\Gamma}\times\ZZ_{\ge 0}^{m_\Gamma}$ denote a state of $(\widehat \G,K_S)$, where the entries of $x'$ and $\widehat{x}$ refer to the original and fictitious species, respectively. Let $e_y$ denote the vector of $\ZZ_{\ge 0}^{m_\Gamma}$ whose entry relative to the complex $y$ is 1, and whose other entries are zero. Consider the set
  $\widehat{\Gamma}=\Gamma'\times \{e_y\}_{y\in\C_\Gamma}$. Since $\Gamma'$ is an irreducible component of $(\G_\Gamma, K_S)$, it follows that  $\widehat{\Gamma}$ 
is a closed set for the reaction network $(\widehat{\G}, K_S)$, in the sense that no state outside $\widehat{\Gamma}$ is accessible from within $\widehat{\Gamma}$. 
Furthermore, from the state $(x',e_y)$ only the states $(x'-y+y', e_{y'})$ are accessible, provided that $x'\geq y$ and that there is a path from $y$ to $y'$ in the reaction graph of $\G_\Gamma$, which is equivalent to the existence of a path from $y+S_y$ to $y'+S_{y'}$ in the reaction graph of $\widehat{\G}$. Since finitely many states are accessible from any given state, every irreducible component of $(\widehat \G, K_S)$ has finitely many states. Finally, since $\widehat{\G}$ is weakly reversible, $(x'',e_{y'})$ is accessible from $(x',e_y)$ if and only if $(x',e_y)$ is accessible from $(x'',e_{y'})$. Thus, we have shown that $\widehat \Gamma$ is a union of irreducible components of $(\widehat \G,K_S)$, each of which has finitely many states.

  Now define a measure $\widehat{\mu}$ on $\widehat{\Gamma}$ by
\begin{equation*}
\widehat{\mu}(x',e_y)=\mu'(x')\text{ for all }(x',e_y)\in\widehat{\Gamma}.
\end{equation*}
Clearly, $\widehat \mu$ is positive on every state of $\widehat \Gamma$. 
We now show that $\widehat{\mu}$ is a non-null stationary measure. Noting that $\lambda_{y'+S_{y'} \to y''+S_{y''}}(x',e_y)$ is nonzero and equal to the rate $\lambda_{y \to y''}(x')$ of $(\G_\Gamma,K_S)$ if and only if $y' = y$, for every $(x',e_y)$ in $\widehat{\Gamma}$ we have
  \begin{align*}
   \sum_{y''\in\widehat{\C}} \sum_{y'\in\widehat{\C}}\widehat{\mu}(x',e_y)\lambda_{y'+S_{y'} \to y''+S_{y''}}(x',e_y)&=\sum_{y''\in\widetilde\C}\mu'(x')\lambda_{y\to y''}(x')\\
   &=\sum_{y''\in\widetilde\C}\mu'(x'+y''-y)\lambda_{y''\to y}(x'+y''-y)\\
   &=\sum_{y''\in\widehat{\C}} \sum_{y'\in\widehat{\C}} \widehat{\mu}(x'+y''-y',e_{y+y''-y'})\lambda_{y''+S_{y''}\to y'+S_{y'}}(x'+y''-y',e_{y+y''-y'})
  \end{align*}
  where in the second equality, we used that $\mu'$ is a complex balanced measure on $\Gamma'$. This shows that $\widehat \mu$ is stationary. 

Let $L$ be an irreducible component within $\widehat{\Gamma}$. Since $\widehat \mu$ is a non-null, $\sigma$-finite stationary measure on $L$ which has only a finite number of states, it follows that $\widehat \mu$ is finite on $L$, and therefore $\widehat \mu$ is a constant multiple of the stationary distribution on $L$. As noted earlier in the proof, every deterministic mass action system defined on the network $\widehat \G$ is complex balanced. Let $c$ be a positive complex balanced equilibrium of the mass action system $(\widehat \G,K_D)$. Hence, by \cite[Theorem 4.1]{anderson:product-form} (stated here as Theorem \ref{thm:anderson}) we have that within every irreducible component $L\subset \widehat{\Gamma}$
\begin{equation*}
\widehat{\mu}(x',e_y)=M_L\frac{c^{x'}}{{x'}!}\text{ for all }(x',e_y)\in L,
\end{equation*}
for some $M_L >0$. 

We now show that $M_L$ does not depend on the irreducible component $L$. First, it follows from $\widehat{\mu}(x',e_y)=\mu'(x')$ that $M_L=\mu(x'){x'}!/c^{x'}=:f(x')$. To show that $M_{L_1}=M_{L_2}$ for two different irreducible components $L_1, L_2 \subset \widehat \Gamma$, it suffices to show that $f(x'_1)=f(x'_2)$ for $x'_1$ and $x'_2$ such that $(x'_1,e_y) \in L$ and $(x'_2,e_{y'}) \in L'$, for some $y,y'\in\C_\Gamma$. Since $x'_1, x'_2 \in \Gamma'$, $x'_2$ is accessible from $x'_1$, i.e. there is a finite sequence of reactions that takes $x'_1$ to $x'_2$. Therefore, to argue by induction, it is sufficient to show that $f(x'_1)=f(x'_2)$ assuming that $x'_2=x'_1-y+y''$ and $\lambda_{y\to y''}(x'_1)>0$. Under the last assumption $(x'_1,e_y)$ and $(x'_2,e_{y''})$ are in the same irreducible component $L$, so necessarily $f(x'_1)=M_L=f(x'_2)$. 
Thus, we have that for every $x' \in \Gamma'$, $\mu'(x')=Mc^{x'}/{x'}!$ and therefore, 
\begin{equation*}
\sum_{x\in\Gamma}\mu(x)=\sum_{x'\in\Gamma'}\mu'(x')=M\sum_{x'\in\Gamma'}\frac{c^{x'}}{{x'}!}\leq M\sum_{x'\in\ZZ^{n_\Gamma}_{\ge 0}}\frac{c^{x'}}{{x'}!}=M\prod_{i=1}^{n_\Gamma} e^{c_i}<\infty.
\end{equation*}
Thus $\mu$ is finite, which concludes the proof. 
\end{proof}

\begin{remark}\label{rem:rvb_measure_not_distr}
Since a reaction balanced measure is complex balanced, a similar result as Proposition \ref{thm:exist_stat_dist} can be stated for a reaction balanced measure. Unlike the case of reaction balance and complex balance, a $\sigma$-finite reaction vector balanced measure within an irreducible component is not necessarily finite, as shown in Remark \ref{rem:no_one_to_many}(1).
\end{remark}

We are now ready to prove Theorem \ref{thm:stoch_bal}.
\begin{proof}[Proof of Theorem \ref{thm:stoch_bal}]
 Suppose that $\mu$ is a $\sigma$-finite complex balanced measure within $\Gamma$, an active irreducible component of $(\G,K_S)$. By Proposition \ref{thm:exist_stat_dist}, $\mu(\Gamma)< \infty$, and therefore $\mu$ is a stationary distribution, up to a positive multiplicative constant. 
The claim is then proved by Theorem \ref{thm:one_to_many_cb_dist}.

Suppose now that $\mu$ is a $\sigma$-finite reaction balanced measure within $\Gamma$, an active irreducible component of $(\G,K_S)$. Since a reaction balanced measure is complex balanced, we have already shown that there exists a unique stationary distribution within every irreducible component $\Gamma'$ of $(\G,K_S)$,
and by \cite[Corollary 19]{cappelletti:complex_balanced} (stated here as Theorem \ref{thm:complex-balance}) they are of the form $\pi_{\Gamma'}(x)=M^{c}_{\Gamma'} c^x/x!$, where $c>0$ and $M^{c}_{\Gamma'}>0$. In particular, $\pi_{\Gamma}=\mu$ is a reaction balanced measure, hence by plugging $\pi_{\Gamma}(x)=M^{c}_{\Gamma} c^x/x!$ and the mass action reaction rates in \eqref{eq:rxnbal} and by simplifying, we obtain
\begin{equation*}
\left( \kappa_{y \to y'} - \kappa_{y' \to y} c^{y'-y} \right) \one_{\{x \ge y\}} = 0
\end{equation*}
for all $x \in \Gamma$ and every $y\to y' \in \R$. Since $\Gamma$ is an active irreducible component, every reaction $y \to y' \in \R$ is active at some state $x$, i.e.\ for every $y \to y' \in \R$ there is a state $x \in \Gamma$ such that $x \ge y$. Therefore, $\kappa_{y \to y'} = \kappa_{y' \to y} c^{y'-y}$ for every $y \to y' \in \R$. It follows that for every irreducible component $\Gamma'$, for all $x \in \Gamma'$ and all $y\to y' \in \R$, we have $\pi_{\Gamma'}(x)\lambda_{y\to y'}(x)=\pi_{\Gamma'}(x+y'-y)\lambda_{y\to y'}(x+y'-y)$, which concludes the proof.
\end{proof}

The following results explore sufficent conditions for reaction balance, given in terms of complex balance and reaction vector balance.

\begin{theorem} \label{thm:rvb+cb=rb}
Suppose that $\pi$ is a $\sigma$-finite complex balanced measure of a stochastic mass action system $(\G,K_S)$. Assume that $\supp(\pi)$ includes a subset $A\subseteq\ZZ_{\geq 0}^n$ 
such that no nonzero polynomial of degree at most $\max_{y\to y'\in\R}\|y\|_1$ vanishes on $A$. Furthermore, suppose that \eqref{eq:rxnvecbal} holds for $\xi \in \RR^n$ and $x \in A$. 
Then, $(\G,K_S)$ is reaction balanced.
\end{theorem}
The proof of Theorem \ref{thm:rvb+cb=rb} makes use of results developed in Section \ref{sec:bridge}. Hence, it is deferred to Subsection \ref{sec:proof_prop}. The following result is an immediate corollary.  
\begin{corollary} \label{cor:rvb+cb=rb}
Suppose that $\pi$ is a $\sigma$-finite complex balanced measure of a stochastic mass action system $(\G,K_S)$ with $\supp(\pi)\subseteq \ZZ^n_{\geq0}$. Assume that $\pi$ is also reaction vector balanced and such that no nonzero polynomial of degree at most $\max_{y\to y'\in\R}\|y\|_1$ vanishes on $\supp(\pi)$.
Then, $(\G,K_S)$ is reaction balanced and cycle balanced.
\end{corollary}
We give here an example that shows that the somewhat technical conditions in the statement of Theorem \ref{thm:rvb+cb=rb} and of Corollary \ref{cor:rvb+cb=rb} are necessary. 
\begin{example}
Consider the stochastic mass action system $(\G,K_S)$ depicted below.
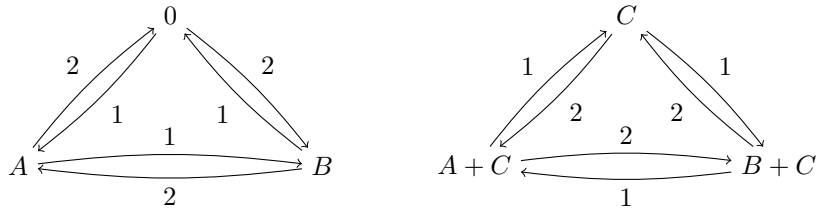
\begin{equation*}\label{cb+rvb+notrb}
  \begin{tikzpicture}[baseline={(current bounding box.center)}]
   \node[state] (1)  at (0,0)  {$A$};
   \node[state] (2)  at (2,2)  {$0$};
   \node[state] (3)  at (4,0)  {$B$};
    \node[state] (4)  at (6,0)  {$A+C$};
     \node[state] (5)  at (8,2)  {$C$};
     \node[state] (6)  at (10,0)  {$B+C$};
   \path[->]
   (1) edge[bend left] node {$2$} (2)
	edge[bend left] node {$1$} (3)
	  (2) edge[bend left] node {$2$} (3)
	edge[bend left] node {$1$} (1)
	  (3) edge[bend left] node {$2$} (1)
	edge[bend left] node {$1$} (2)
	   (4) edge[bend left] node {$2$} (6)
	edge[bend left] node {$1$} (5)
	  (5) edge[bend left] node {$2$} (4)
	edge[bend left] node {$1$} (6)
	  (6) edge[bend left] node {$2$} (5)
	edge[bend left] node {$1$} (4);
  \end{tikzpicture}
 \end{equation*}
$\pi(x) = (x_A!x_B!)^{-1}$ for $x\in\ZZ^3_{\geq0}$ is a complex balanced measure with $\supp(\pi)=\ZZ^3_{\geq0}$, thus $(\G,K_S)$ is complex balanced by Proposition \ref{prop:nonexplosive} and Remark \ref{rem:nonexplosive}. The irreducible components of the stochastic mass action system $(\G,K_S)$ contained in $\ZZ^n_{\geq 0}$ are given by $x_C=l$ with $l\in\ZZ_{\geq0}$. It can be checked that the stationary distribution $\pi_1$ within the irreducible component with $l=1$ is reaction vector balanced, but no stationary distribution of $(\G,K_S)$ is reaction balanced. This does not contradict Theorem \ref{thm:rvb+cb=rb} because the nonzero polynomial $x_C-1$ (of degree $1 < 2 = \max_{y\to y'\in\R}\|y\|_1$) vanishes on $\supp(\pi_1)$.
\end{example}

The following result is somewhat surprising, as it breaks down the symmetry between the results for graphical balance in the deterministic sense (Section \ref{sec:equilibria_deterministic}) and the results for graphical balance in the stochastic sense. 
Indeed, the corresponding statement for the deterministic case is not true in general (see Remark \ref{rem:no_implication_det}(1)).

\begin{corollary}[Complex balance $\&$ reaction vector balance $ \implies$  Reaction balance] \label{cor:cb_plus_rvb_rb_2}
Suppose that $(\G,K_S)$ is a complex balanced stochastic mass action system. If $(\G,K_S)$ is reaction vector balanced, then $\G$ is reversible and $(\G,K_S)$ is reaction balanced. 
\end{corollary}
\begin{proof}
 By Theorem \ref{thm:nec_conds_measure}, since there exists a complex balanced distribution on an active irreducible component, then $\G$ is weakly reversible. Hence, due to \cite{craciun:dynamical}, $\ZZ^n_{\ge 0}$ is a union of irreducible components. Therefore, there exists a complex balanced distribution $\pi$ with $\supp(\pi)=\ZZ^n_{\ge 0}$, which is necessarily reaction vector balanced by assumption. The conclusion follows from Corollary \ref{cor:rvb+cb=rb}.
\end{proof}

\section{Bridging deterministic and stochastic mass action systems}\label{sec:bridge}

 The main focus of this section is to relate graphical balance of a deterministic mass action system with graphical balance of the corresponding stochastic mass action system, and vice versa. We begin by stating some known results.  
 \subsection{Existing results}  
Connection between equilibria of particular deterministic mass action systems and the stationary distributions of the stochastic counterparts have been shown in \cite{anderson:product-form, joshi:detailed, cappelletti:complex_balanced}. 

The first result in this sense is \cite[Theorem 4.1]{anderson:product-form}, which states the following:

\begin{theorem}\label{thm:anderson}
Let $(\G,K_D)$ be a deterministic mass action system with a positive complex balanced equilibrium $c \in \RR^n_{> 0}$. Then, the corresponding stochastic mass action system $(\G,K_S)$ has a product-form Poisson-like stationary distribution within every irreducible component $\Gamma$ given by the expression
 \begin{equation}\label{eq:product-form}
  \pi_\Gamma(x)=M^{c}_\Gamma\frac{c^x}{x!}\quad\text{for }x\in\Gamma,
 \end{equation}
 where $M^{c}_\Gamma$ is a normalizing constant.
\end{theorem}

The study is then carried on in \cite{joshi:detailed, cappelletti:complex_balanced}, where further relations between stochastic and deterministic models are unveiled. We start with presenting one of the main results of \cite[Theorems 5.9 and 5.10]{joshi:detailed}:
\begin{theorem}\label{thm:badal}
 If the deterministic mass action system $(\G,K_D)$ is reaction balanced (i.e. detailed balanced as a reaction network), then the corresponding stochastic mass action system $(\G,K_S)$ is reaction vector balanced (i.e. detailed balanced as a Markov chain). Moreover, the converse holds if the function
 $$y\to y'\quad \longmapsto\quad y'-y$$
 is a one-to-one correspondence between the reactions in $\R$ and their reaction vectors.
\end{theorem}
In \cite{cappelletti:complex_balanced} a study on complex balanced distribution is conducted, and a stochastic counterpart for the deterministic model deficiency zero theory is developed. In particular, \cite[Corollary 19]{cappelletti:complex_balanced} states the following:

\begin{theorem} \label{thm:complex-balance}
   If a stochastic reaction system $(\G,K_S)$ admits a complex balanced distribution within an active irreducible component then $\G$ is weakly reversible. Moreover, a stochastic mass action system $(\G,K_S)$ admits a complex balanced distribution within an active irreducible component if and only if the corresponding deterministic mass action system $(\G,K_D)$ admits a positive complex balanced state. If this is the case, then on every irreducible component $\Gamma$ there exists a unique stationary distribution $\pi_\Gamma$. Such $\pi_\Gamma$ is a complex balanced distribution and it has the form \eqref{eq:product-form}, where $c$ is a positive complex balanced equilibrium of $(\G,K_D)$.
\end{theorem}

 Moreover, the following is implied by \cite[Theorem 23 and discussion in Section 5.1]{cappelletti:complex_balanced}.
 \begin{theorem} \label{thm:poisson_converse}
Let $A \subseteq \ZZ^n_{\ge 0}$ be any set such that no nonzero polynomial of degree at most $\max_{y\to y'\in\R}\|y\|_1$ vanishes on $A$. Let $\widetilde A := A \cup \{x \in \ZZ^n_{\ge 0} : x \ge y \mbox{  for some }y \to y' \in \R \mbox{ and such that } x-y+y' \in A \}$. Suppose that $\pi$ is a stationary distribution of a stochastic mass action system $(\G,K_S)$ such that $\pi(x)=M(x)c^x/x!$ for $x\in \widetilde A$, where $M(x) = M(x')$ for any $x, x'$ within the same irreducible component. Then $(\G,K_D)$ is complex balanced and $c$ is a complex balanced equilibrium of $(\G,K_D)$.
\end{theorem}

\begin{corollary}
 Suppose that $\pi=c^x/x!$ for $x\in\ZZ^n_{\ge0}$ is a stationary measure of a stochastic mass action system $(\G,K_S)$. Then $(\G,K_D)$ is complex balanced and $c$ is a complex balanced equilibrium of $(\G,K_D)$.
\end{corollary}

\subsection{Expanding the bridge} \label{sec:expand_bridge}

In this section, we establish further connections between a deterministic mass action system $(\G,K_D)$ and the corresponding stochastic mass action system $(\G,K_S)$. 

\begin{theorem} \label{thm:bridge}
Suppose that $(\G,K_D)$ is a deterministic mass action system and $(\G,K_S)$ is the corresponding stochastic mass action system.
\been
\item $(\G,K_D)$ is reaction balanced if and only if $(\G,K_S)$ is reaction balanced. 
\item $(\G,K_D)$ is complex balanced if and only if $(\G,K_S)$ is complex balanced. 
\item $(\G,K_D)$ is cycle balanced if and only if $(\G,K_S)$ is cycle balanced.
\enen
\end{theorem}
\begin{proof}
\been
\item Suppose that $(\G,K_D)$ is reaction balanced. Let $c$ be a positive reaction balanced equilibrium of $(\G,K_D)$. In other words, for every $y \to y' \in \R$, the following holds:
\begin{equation*}
\kappa_{y \to y'}c^y = \kappa_{y' \to y} c^{y'}
\end{equation*}
Then $c$ is a complex balanced equilibrium of $(\G,K_D)$ and by Theorem \ref{thm:anderson}, $\pi(x) = c^x/ x!$ for all $x \in \ZZ^n_{\ge 0}$ defines a complex balanced measure of $(\G,K_S)$. It is straightforward to verify by a direct calculation that $\pi$ is a reaction balanced stationary measure of $(\G,K_S)$. On the other hand, if $(\G,K_S)$ is reaction balanced, by Theorems \ref{thm:stoch_rel} and \ref{thm:complex-balance} there is a $c>0$ such that $\pi(x) = c^x/x!$ for all $x \in \ZZ^n_{\ge 0}$ defines a reaction balanced stationary measure. But this implies that $c$ is a reaction balanced equilibrium of $(\G,K_D)$ which proves that $(\G, K_D)$ is reaction balanced (by Theorem \ref{thm:bal}). 

\item This follows from Theorem \ref{thm:complex-balance}. 

\item In the case of mass action kinetics, both the deterministic cycle balance condition \eqref{def:cyc_bal} for a positive vector $c$ and the stochastic cycle balance condition \eqref{eq:cyc_bal_m} for a distribution $\pi$ on an active irreducible component reduce to the same condition on reaction rate constants. Both $(\G,K_D)$ and $(\G,K_S)$ are cycle balanced if and only if for every sequence of distinct complexes $(y_1, \ldots , y_j) \subseteq \C$ with $j\geq3$, the  reaction rate constants satisfy 
\begin{equation*}
\prod_{i=1}^j \kappa_{y_i \to y_{i+1}} = \prod_{i=1}^j \kappa_{y_{i+1} \to y_{i}}
\end{equation*}
where by assumption $\kappa_{y \to y'} = 0$ if $y \to y' \notin \R$, and $y_{j+1}:=y_1$. To see this in the stochastic setting, note that if an irreducible component $\Gamma$ is active, for $y_1\to y_2\in\R$ there exists $x\in\Gamma$ with $x\geq y_1$ and for all $1 \le i \le j$, if $x\geq y_i$, then $x-y_i+y_{i+1}\geq y_{i+1}$. It follows that $(\G,K_D)$ is cycle balanced if and only if $(\G,K_S)$ is cycle balanced. 
\enen
\end{proof}
\begin{remark}
 Theorem \ref{thm:bridge} implies that we can talk about complex balanced, reaction balanced and cycle balanced mass action systems regardless whether we are considering the stochastic or deterministic modeling regime.
\end{remark}

\begin{remark} \label{ex:srvb_drvb}
In general, the corresponding statement in Theorem \ref{thm:bridge} for reaction vector balance is not true in either direction. We give examples to illustrate this point. 
\been
\item\label{item:example} (Deterministic reaction vector balance $\centernot \implies$ Stochastic reaction vector balance.) Denote the deterministic mass action system in \eqref{eq:square} by $(\G,K_D)$. We saw earlier that $(\G,K_D)$ is reaction vector balanced and complex balanced but not reaction balanced. By Theorem \ref{thm:bridge}, it follows that the corresponding stochastic mass action $(\G,K_S)$ is complex balanced. Now suppose that $(\G,K_S)$ is reaction vector balanced, so that by Corollary \ref{cor:cb_plus_rvb_rb_2}, $(\G,K_S)$ is reaction balanced. This implies that, by Theorem \ref{thm:bridge}, $(\G,K_D)$ is reaction balanced, which is a contradiction. Therefore, $(\G,K_S)$ cannot be reaction vector balanced. 

Another example is given by the deterministic mass action system \eqref{eq:ACR}, which is reaction vector balanced. However, the corresponding stochastic mass action system has no active irreducible components and therefore no stationary distributions within an active irreducible component. Hence, it cannot be reaction vector balanced.
\item \label{ex:srvb_drvb_pt} (Stochastic reaction vector balance $\centernot \implies$ Deterministic reaction vector balance.) Consider the stochastic mass action system $(\G,K_S)$ depicted below:

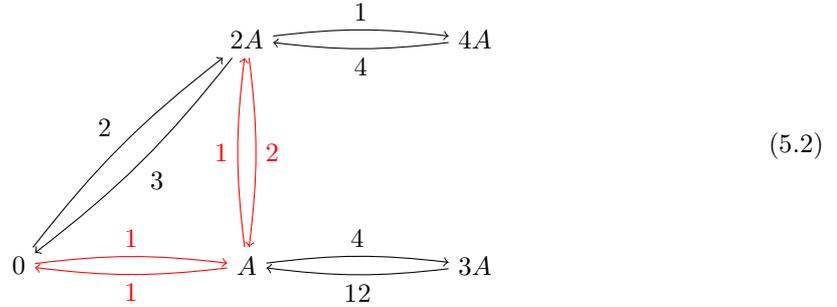
\begin{equation}\label{eq:stochrvb}
  \begin{tikzpicture}[baseline={(current bounding box.center)}]
   \node[state] (1)  at (0,0)  {$0$};
   \node[state] (2)  at (3,0)  {$A$};
   \node[state] (3)  at (6,0)  {$3A$};
    \node[state] (4)  at (3,3)  {$2A$};
     \node[state] (5)  at (6,3)  {$4A$};
   \path[->]
    (1) edge[red,bend left] node {$1$} (2)
    	edge[bend left] node {$2$} (4)
    (2) edge[red,bend left] node {$1$} (1)
         edge[bend left] node {$4$} (3)
         edge[red,bend left] node {$1$} (4)
    (3) edge[bend left] node {$12$} (2)
    (4) edge[red,bend left] node {$2$} (2)
          edge[bend left] node {$3$} (1)
          edge[bend left] node {$1$} (5)
    (5) edge[bend left] node {$4$} (4);
  \end{tikzpicture}
 \end{equation}
 Similarly to Remark \ref{rem:no_one_to_many} (2), the reactions corresponding to the edges in red have reaction vector $\pm1$ while the reactions corresponding to the edges in black have reaction vector $\pm2$. 
 The only active irreducible component is $\ZZ_{\ge 0}$, and the only stationary distribution within $\ZZ_{\ge 0}$ is
 \begin{equation*}
 \pi(x) = M\prod_{j=0}^{x-1}\frac{1}{2j+1}\quad\text{for }x\in\ZZ_{\geq0},
 \end{equation*}
where $M$ is a normalizing constant. It can be checked that $\pi$ is reaction vector balanced, which implies that $(\G,K_S)$ is reaction vector balanced. The reaction vector balanced equilibria of the corresponding deterministic mass action system $(\G,K_D)$ are solutions of 
\begin{align*}
&0 = 1 - 2a^2, \\
&0 =  4 + 2a + 2a^2 - 24a^3 - 8a^4.   
\end{align*}
The only positive solution of the first equation is $a = 1/ \sqrt{2}$, which is not a solution of the second equation. Thus $(\G,K_D)$ does not have any reaction vector balanced equilibria. 
\enen
\end{remark}

As an immediate consequence of Corollary \ref{cor:cb_plus_rvb_rb_2} and Theorem \ref{thm:bridge}, the following can be stated, in the spirit of Theorem \ref{thm:badal} (due to \cite{joshi:detailed}):

\begin{corollary}\label{cor:cb_plus_detrb}
  If the stochastic mass action system $(\G,K_S)$ is reaction vector balanced (i.e.\ detailed balanced as a Markov chain) and complex balanced, then the corresponding deterministic mass action system $(\G,K_D)$ is reaction balanced (i.e.\ detailed balanced as a reaction system).
\end{corollary}

Note that the converse does not hold: deterministic complex balance and reaction vector balance does not imply stochastic reaction vector balance (hence reaction balance), as highlighted in Remark \ref{ex:srvb_drvb}\eqref{item:example}. 

  Our results complete the framework of \cite{anderson:product-form, joshi:detailed, cappelletti:complex_balanced}. Indeed, given a stochastic mass action system $(\G,K_S)$, we say that a set $A \subseteq \ZZ^n_{\ge 0}$ has property $P$ if no nonzero polynomial of degree at most $\max_{y\to y'\in\R}\|y\|_1$ vanishes on $A$. Then, we have the following scheme of implications:
  
 
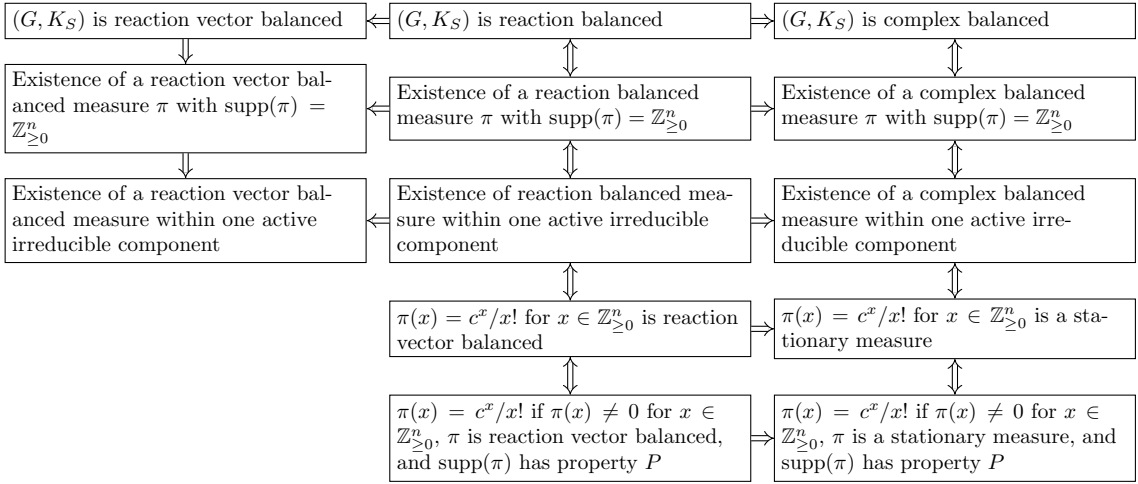
\begin{center}
   \begin{tikzpicture}[auto, every node/.style={scale=0.8}]
        \tikzstyle{block} = [draw, rectangle, text width=5.7cm];
        \node [block] (det_det_bal) {$(G,K_S)$ is reaction balanced};
        \node [block, below=0.5cm of det_det_bal] (whittle_det_bal_all) {Existence of a reaction balanced measure $\pi$ with $\supp(\pi)=\ZZ^n_{\ge 0}$};
        \node [block, below=0.5cm of whittle_det_bal_all] (whittle_det_bal_one) {Existence of reaction balanced measure within one active irreducible component};
        \node [block, below=0.5cm of whittle_det_bal_one] (stoch_det_bal) {$\pi(x)=c^x/x!$ for $x\in\ZZ^n_{\ge 0}$ is reaction vector balanced};
        \node [block, below=0.5cm of stoch_det_bal] (stoch_det_bal_A) {$\pi(x)=c^x/x!$ if $\pi(x) \ne 0$ for $x\in\ZZ^n_{\ge 0}$, $\pi$ is reaction vector balanced, and $\supp(\pi)$ has property $P$};
        \node [block, right=0.3cm of det_det_bal] (det_compl_bal) {$(G,K_S)$ is complex balanced};
        \node [block, right=0.3cm of whittle_det_bal_all] (stoch_compl_bal_all) {Existence of a complex balanced measure $\pi$ with $\supp(\pi)=\ZZ^n_{\ge 0}$};
        \node [block, right=0.3cm of whittle_det_bal_one] (stoch_compl_bal_one) {Existence of a complex balanced measure within one active irreducible component};
        \node [block, right=0.3cm of stoch_det_bal] (poisson) {$\pi(x)=c^x/x!$ for $x\in\ZZ^n_{\ge 0}$ is a stationary measure};
        \node [block, right=0.3cm of stoch_det_bal_A] (poisson_A) {$\pi(x)=c^x/x!$ if $\pi(x) \ne 0$ for $x\in\ZZ^n_{\ge 0}$, $\pi$ is a stationary measure, and $\supp(\pi)$ has property $P$};
        \node [block, left=0.3cm of det_det_bal] (rvb) {$(G,K_S)$ is reaction vector balanced};
        \node [block, left=0.3cm of whittle_det_bal_all] (rvb_all) {Existence of a reaction vector balanced measure $\pi$ with $\supp(\pi)=\ZZ^n_{\ge 0}$};
        \node [block, left=0.3cm of whittle_det_bal_one] (rvb_one) {Existence of a reaction vector balanced measure within one active irreducible component};
     \draw[implies-implies,double equal sign distance] (det_det_bal)--(whittle_det_bal_all);
     \draw[implies-implies,double equal sign distance] (whittle_det_bal_all)--(whittle_det_bal_one);
     \draw[implies-implies,double equal sign distance] (whittle_det_bal_one)--(stoch_det_bal);
     \draw[implies-implies,double equal sign distance] (stoch_det_bal)--(stoch_det_bal_A);
     \draw[implies-implies,double equal sign distance] (det_compl_bal)--(stoch_compl_bal_all);
     \draw[implies-implies,double equal sign distance] (stoch_compl_bal_all)--(stoch_compl_bal_one);
     \draw[implies-implies,double equal sign distance] (stoch_compl_bal_one)--(poisson);
     \draw[implies-implies,double equal sign distance] (poisson)--(poisson_A);
     \draw[-implies,double equal sign distance] (rvb)--(rvb_all);
     \draw[-implies,double equal sign distance] (rvb_all)--(rvb_one);
     \draw[-implies,double equal sign distance] (det_det_bal)--(det_compl_bal);
     \draw[-implies,double equal sign distance] (whittle_det_bal_one)--(stoch_compl_bal_one);
     \draw[-implies,double equal sign distance] (whittle_det_bal_all)--(stoch_compl_bal_all);
     \draw[-implies,double equal sign distance] (stoch_det_bal)--(poisson);
     \draw[-implies,double equal sign distance] (stoch_det_bal_A)--(poisson_A);
     \draw[-implies,double equal sign distance] (det_det_bal)--(rvb);
     \draw[-implies,double equal sign distance] (whittle_det_bal_one)--(rvb_one);
     \draw[-implies,double equal sign distance] (whittle_det_bal_all)--(rvb_all);
     
   \end{tikzpicture} 
 \end{center}


Further implications connecting the deterministic and stochastic models are summarized in Figure \ref{summary_of_implications}.

\subsection{Proof of Theorem \ref{thm:rvb+cb=rb}}\label{sec:proof_prop}
 We have $A\subseteq\supp(\pi)$ with $A$ being a non-empty subset of $\ZZ_{\geq 0}^n$. Moreover, since complex balanced measure are stationary measures by Theorem \ref{thm:stoch_imp}, by standard Markov chain theory and by the fact that under mass action kinetics the irreducible components are either contained in $\ZZ^n_{\geq0}$ or disjoint from it, we can restrict $\pi$ to a complex balanced measure within $\ZZ_{\geq 0}^n$.
  
 Since $\pi$ is a $\sigma$-finite complex balanced measure of $(\G,K_S)$, by Proposition \ref{thm:exist_stat_dist}, $\pi(\ZZ_{\geq 0}^n) < \infty$. Fix $y\to y'\in\R$. Since $(\G,K_S)$ is a mass action system, we have $\lambda_{y \to y'}(x) = \kappa_{y\to y'}\prod_{i=1}^n x_i(x_i-1)\cdots(x_i-y_i+1)$ for some $\kappa_{y\to y'} > 0$ and for all $x\in\ZZ^n_{\geq0}$ (see Remark \ref{rem:mass_action_no_indicator}). By assumption, $\lambda_{y \to y'}$ cannot vanish on $\supp(\pi)$, so there exists an irreducible component in $\supp(\pi)$ such that $y\to y'$ is active at some $x\in\Gamma$. Since $\pi$ is complex balanced, by Theorem \ref{thm:stoch_bal} $\G_{\Gamma}$ is weakly reversible. It also follows from \cite[Theorem 18]{cappelletti:complex_balanced} that $(\G_\Gamma, K_D)$ is (deterministically) complex balanced with the reaction rates inherited from  $(\G,K_D)$. Since this holds for every $y\to y'\in\R$, it follows that $\G$ is weakly reversible. Therefore, the subnetworks $\G_{\Gamma}$ are necessarily a union of connected components of $\G$: if $x\geq y$, then $x-y+y'\geq y'$, so if $y\to y'$ is active at some state of $\Gamma$, the same holds for $y'\to y''$, and so on. In conclusion, since \eqref{eq:compbal} only concerns reactions in the same connected component, $(\G, K_D)$ is complex balanced.
 By Theorem \ref{thm:anderson}, for all $x \in \supp(\pi)$, $\pi(x) = M_\Gamma c^x/x!$, where $M_\Gamma$ is a positive constant depending on the irreducible component containing $x$, and $c \in \ZZ^n_{> 0}$. By assumption, \eqref{eq:rxnvecbal} holds for all $x\in A$ and all reaction vectors $\xi$. Substituting the expressions for $\pi$ and $\lambda$ into the reaction vector balance condition \eqref{eq:rxnvecbal} and simplifying yields
\begin{equation*}
\sum_{y\to y'\,:\,y' - y = \xi} \left( \kappa_{y \to y'} - c^\xi \kappa_{y' \to y} \right) \prod_{i=1}^n x_i(x_i-1)\cdots(x-y_i+1) = 0
\end{equation*}
for all $x\in A$ and all reaction vectors $\xi$. It follows that the polynomial on the left-hand side is null. By linear independence of the polynomials $\prod_{i=1}^nx_i(x_i-1)\cdots(x-y_i+1)$ in the sense of Remark \ref{rem:mass_action_no_indicator}, the expressions $\kappa_{y \to y'} - c^{y'-y} \kappa_{y' \to y}$ must be equal to zero for all $y\to y'\in\R$. This implies that $c$ is a reaction balanced state and therefore by Theorem \ref{thm:bridge}, $\pi$ is a reaction balanced measure. \hfill\qed

\subsection*{Acknowledgments}
We thank David Anderson and Gheorghe Craciun, for organizing the Workshop on Mathematics of Reaction Networks at University of Wisconsin-Madison in October 2015. The project was conceived during the workshop.

\end{document}